\documentclass[a4paper,12pt,twoside]{article}

\usepackage{doi}

\usepackage[dvipsnames]{xcolor} 

\usepackage{hyperref} 
\hypersetup{breaklinks, colorlinks,
    linkcolor = {Blue},
    citecolor = {Blue},
    urlcolor  = {Blue}
}

\usepackage{amsmath,amssymb}
\usepackage{mathtools} 
\usepackage{alphabeta} 

\usepackage{amsthm, 
}

\usepackage[numbers,sort&compress]{natbib}
\usepackage[text={6.5in,9.5in},centering]{geometry}

\usepackage[inline,shortlabels]{enumitem} 
\renewlist{enumerate}{enumerate}{2}
\setlist[enumerate,1]{label=\textup{(\alph*)},
ref={\alph*}, align=left, labelsep=0.5ex, leftmargin=*}
\setlist[enumerate,2]{label=\textup{({\roman*})},
ref={\roman*}, align=right, labelsep*=1ex, widest={(ii)},  
leftmargin=5.4ex}

\DeclareSymbolFont{bbold}{U}{bbold}{m}{n}
\DeclareMathSymbol{\bbone}{\mathord}{bbold}{"31}

\theoremstyle{plain}
\newtheorem{theorem}{Theorem}[section]
\newtheorem{corollary}{Corollary}[section]
\newtheorem{lemma}{Lemma}[section]
\newtheorem{proposition}{Proposition}[section]

\newtheoremstyle{boldremex}
    {\dimexpr\topsep/2\relax} 
    {\dimexpr\topsep/2\relax} 
    {}          
    {}          
    {\bfseries} 
    {.}         
    {.5em}      
    {}          
    
\theoremstyle{boldremex}
\newtheorem{remark}{Remark}[section]
\newtheorem{example}{Example}[section]

\makeatletter
\renewenvironment{proof}[1][\proofname]{%
   \par\pushQED{\qed}\normalfont%
   \topsep6\p@\@plus6\p@\relax
   \trivlist\item[\hskip\labelsep\bfseries#1\@addpunct{.}]%
   \ignorespaces
}{%
   \popQED\endtrivlist\@endpefalse
}
\makeatother

\newcommand{\CC}{\mathbb{C}}
\newcommand{\RR}{\mathbb{R}}
\newcommand{\NN}{\mathbb{N}}
\newcommand{\ZZ}{\mathbb{Z}}
\newcommand{\Zpl}{\ZZ_+}

\newcommand{\calf}{\mathcal{F}}
\newcommand{\calm}{\mathcal{M}}

\newcommand{\calfw}{\calf_{\mathrm{W}}}

\newcommand{\set}[1]{\underline{#1}}

\newcommand{\dd}{{\mathrm{d}}}
\newcommand{\ee}{{\mathrm{e}}}
\newcommand{\pii}{{\pi}}
 
\newcommand{\abs}[1]{\lvert#1\rvert}              

\newcommand{\ABS}[1]{\Bigl\lvert#1\Bigr\rvert}    

\newcommand{\Be}{\mathrm{Be}}

\newcommand{\bfcdot}{{\boldsymbol{\cdot}}}

\newcommand{\card}[1]{|#1|}                 

\newcommand{\Cov}{\mathrm{Cov}}

\newcommand{\dirac}{\delta}                 

\newcommand{\dtv}[2]{d_{\mathrm{TV}}(#1,#2)}

\newcommand{\EE}{\mathrm{E}} 

\newcommand{\floor}[1]{{\lfloor #1 \rfloor}}

\newcommand{\normone}[1]{\|#1\|_1}
\newcommand{\norminfty}[1]{\|#1\|_\infty}

\newcommand{\normtv}[1]{\|#1\|_{\mathrm{TV}}}  
\newcommand{\Normtv}[1]{\Bigl\|#1\Bigr\|_{\mathrm{TV}}}  

\newcommand{\normloc}[1]{\|#1\|_{\mathrm{loc}}}

\newcommand{\normw}[1]{\|#1\|_{\mathrm{W}}}

\newcommand{\po}{\mathrm{po}}

\newcommand{\Po}{\mathrm{Po}}

\newcommand{\Var}{\mathop{{\mathrm{Var}}}}

\newcommand{\newatop}[2]{\genfrac{}{}{0pt}{}{\scriptstyle #1}{\scriptstyle #2}}

\scrollmode

\begin{document}
\title{On the Poisson approximation of random diagonal sums of Bernoulli 
matrices}
\author{
Bero Roos\footnote{
Postal address:
FB IV -- Mathematics, 
Trier University, 
54286 Trier, Germany. 
E-mail: \texttt{bero.roos@uni-trier.de}}\medskip\\
Trier University 
\date{
}
}
\maketitle
\begin{abstract}
We use the Stein-Chen method to prove new explicit 
inequalities for the total 
variation, Wasserstein and local distances between the distribution of 
a random diagonal sum of a Bernoulli matrix and a Poisson distribution. 
Approximation results using a finite signed measure of higher order 
are given as well. Some of our bounds improve on those in Theorem 4.A of
A.D. Barbour, L. Holst and S. Janson (Poisson approximation. Clarendon 
Press, Oxford, 1992). \medskip 
 
\end{abstract}
\section{Introduction and review of some known results} 
\label{s86256}
Let $n\in\NN=\{1,2,3,\dots\}$ be a natural number 
with $n\neq1$ and
$X_{j,r}$ for $j,r\in\set{n}:=\{1,\dots,n\}$ be independent random 
variables with Bernoulli distributions $P^{X_{j,r}}=\Be(p_{j,r})$ with 
success probabilities $p_{j,r}\in[0,1]$. 
For $k,\ell\in\NN$ with $k\leq \ell$, we let
$\set{\ell}_{\neq}^k=\{(j_1,\dots,j_k)\,\vert\,j_1,\dots,j_k
\in\set{\ell}\text{ pairwise distinct}\}$, where we also write 
$(j(1),\dots,j(k))$ for $(j_1,\dots,j_k)$. 
In particular, $\set{n}_{\neq}^n$ is the set of all 
permutations on $\set{n}$.
Let $\pi=(\pi(1),\dots,\pi(n))$ be a random permutation uniformly 
distributed on $\set{n}_{\neq}^n$. 
We assume that all the $X_{j,r}$'s and $\pi$ are independent. 
Let us call $(X_{1,\pi(1)},X_{2,\pi(2)},\dots,X_{n,\pi(n)})$ 
a random (generalized) diagonal of the Bernoulli matrix $X=(X_{j,r})$
and 
\begin{align*}
S_n=\sum_{j=1}^nX_{j,\pi(j)}
\end{align*}
the corresponding random diagonal sum. 
If the $X_{j,r}$'s are constants, $S_n$ is sometimes also called 
Hoeffding permutation statistic (see \citet{MR3920366}) or 
Hoeffding statistic (see \citet{MR4372142}).  
Further let 
\begin{gather*}
\overline{p}_{j,\bfcdot}=\frac{1}{n}\sum_{r=1}^np_{j,r}
\quad\text{for } j\in\set{n}, \quad 
\overline{p}_{\bfcdot,r}=\frac{1}{n}\sum_{j=1}^np_{j,r}
\quad\text{for } r\in\set{n},\\
\lambda=\EE S_n
=\frac{1}{n}\sum_{j=1}^n\sum_{r=1}^np_{j,r}
=\sum_{j=1}^n\overline{p}_{j,\bfcdot}
=\sum_{r=1}^n\overline{p}_{\bfcdot,r}>0,\quad 
\overline{p}=\frac{\lambda}{n}.
\end{gather*}

In this paper, we consider the approximation of the 
distribution $P^{S_n}$ of $S_n$ by a Poisson distribution
and also by a signed measure of higher order. 
To  measure the accuracy, we use the total variation, 
Wasserstein and local norms, the definitions of which require some 
notation. 
Let $\Zpl=\NN\cup\{0\}$, $\RR$ and $\CC$ 
be the sets of nonnegative integers, real and complex numbers, 
respectively. For two sets $A$ and $B$, let
$B^A$ be the set of functions from $A$ to $B$. 
For $f\in \RR^{\Zpl}$, 
let $\Delta f\in\RR^{\Zpl}$ be defined by 
$\Delta f(m)=f(m+1)-f(m)$ for $m\in\Zpl$ 
and set $\norminfty{f}=\sup_{m\in\Zpl}\abs{f(m)}$ and
$\normone{f}=\sum_{m\in\Zpl}\abs{f(m)}$.
Let $\calfw=\{f\in \RR^{\Zpl}\,\vert\,\norminfty{\Delta f}\leq 1\}$. 
Let $\calm$ be the vector space of all finite signed measures 
on the power set of $\Zpl$ and 
set 
\begin{align*}
\calm'=\Bigl\{Q\in \calm\,\Big\vert\,Q(\Zpl)=0, 
\sum_{m=0}^\infty m\abs{f_Q(m)}<\infty\Bigr\},
\end{align*}
where $f_{Q}\in \RR^{\Zpl}$ with $f_{Q}(m)=Q(\{m\})$ for $m\in\Zpl$ 
is the counting density of $Q$.
Let 
\begin{gather*}
\normtv{Q}
=\normone{f_Q},  \quad
\normloc{Q}
= \norminfty{f_Q},  \\
\dtv{Q_1}{Q_2}=\sup_{A\subseteq \Zpl}\abs{Q_1(A)-Q_2(A)}
\end{gather*}
be the total variation norm and the local norm   
of $Q\in\calm$, and the total variation distance between 
$Q_1,Q_2\in\calm$. 
The Wasserstein norm $\normw{Q}$ 
(sometimes also called Fortet-Mourier or Kantorovich norm) 
of $Q\in\calm'$ can be defined by 
\begin{align*}
\normw{Q}=\sum_{m=0}^\infty\abs{Q(\set{m}\cup\{0\})}.
\end{align*}
It is well-known that, for $Q_1,Q_2\in\calm$ and $Q\in\calm'$,  
\begin{gather}
\dtv{Q_1}{Q_2}
=\frac{1}{2}\normtv{Q_1-Q_2}
\quad\text{ if }Q_1(\Zpl)=Q_2(\Zpl), \nonumber\\
\normw{Q}
=\sup_{f\in\calfw}\ABS{\int f \,\dd Q}, \quad 
\normw{(\dirac_1-\dirac_0)*Q}=\normtv{Q}, \label{e2985839}
\end{gather}
where $\dirac_m\in\calm$ is the Dirac measure at point $m\in\Zpl$, and
$*$ denotes convolution. 
Let $\Po(t)$ be the Poisson distribution with mean $t\in(0,\infty)$.

The literature contains some explicit inequalities for the total 
variation distance between $P^{S_n}$ and the Poisson distribution
with the same mean. We are not aware of any published explicit bounds 
concerning the Wasserstein and local distances in this context. 
However, the local distance was considered by 
\citet[Theorem 2.10]{MR3920366}, but they used a translated Poisson 
distribution to improve the accuracy of approximation and the estimate 
given there is not explicit, since it contains an $O$-term. 

Let us discuss some total variation bounds. 
\citet[Theorem 2.1]{MR385977} adapted Stein's \cite{MR0402873} 
method to prove that, 
for $n\geq 5$, 
\begin{align}
\dtv{P^{S_n}}{\Po(\lambda)}
&\leq 7.875\min\Big\{1,\frac{1}{\sqrt{\lambda}}\Big\}
\Bigl(\sum_{j=1}^n
\overline{p}_{j,\bfcdot}^2
+\sum_{r=1}^n\overline{p}_{\bfcdot,r}^2\Bigr), \label{e228658}\\
\dtv{P^{S_n}}{\Po(\lambda)}
&\leq 22.625 \frac{1}{\lambda}\Bigl(\sum_{j=1}^n
\overline{p}_{j,\bfcdot}^2
+\sum_{r=1}^n\overline{p}_{\bfcdot,r}^2\Bigr).\label{e92163845}
\end{align}
In \citet[Theorem 7.1]{MR0980737} and 
\citet[Theorem 4.A on page 78]{MR1163825} 
the Stein-Chen method and coupling is used to 
show refinements of \eqref{e228658} and \eqref{e92163845} 
in the case that the matrix $(p_{j,r})$ is in $\{0,1\}^{\set{n}\times \set{n}}$. 
We only state the result in \cite{MR1163825}, which improves on that in 
\cite{MR0980737}. It says that
\begin{align}
\dtv{P^{S_n}}{\Po(\lambda)}
&\leq \frac{1-\ee^{-\lambda}}{\lambda}
\Bigl(\frac{n-2}{n}(\lambda-\Var S_n)
+\frac{2\lambda^2}{n}\Bigr) \label{e17514745}\\
&\leq \frac{3}{2}\frac{1-\ee^{-\lambda}}{\lambda}
\Bigl(\sum_{j=1}^n\overline{p}_{j,\bfcdot}^2
+\sum_{r=1}^n\overline{p}_{\bfcdot,r}^2-\frac{2\lambda}{3n}\Bigr).
\label{e346689}
\end{align}
As mentioned in Remark 4.1.3 of \cite{MR1163825}, the proof of 
Theorem 4.A in \cite{MR1163825} can be adapted to prove 
\eqref{e92163845} in the general case 
$(p_{j,r})\in[0,1]^{\set{n}\times \set{n}}$ 
with constant $3/2$ in place of the constant $22.625$. 
Our results below imply that \eqref{e17514745} also holds in the case 
$(p_{j,r})\in[0,1]^{\set{n}\times \set{n}}$; on the other hand, 
\eqref{e346689} has to be slightly adapted, see Remark \ref{r24672345}. 

Under the additional assumption that $p_{j,r}=\bbone_{[-1,a_j]}(r)$ 
for all $j,r\in\set{n}$ with $a_1,\dots,a_n\in\set{n}\cup\{-1,0\}$, 
\eqref{e17514745} can be improved to 
\begin{align}\label{e1721975}
\dtv{P^{S_n}}{\Po(\lambda)}
&\leq (1-\ee^{-\lambda})\Bigl(1-\frac{\Var S_n}{\lambda}\Bigr).
\end{align}
Here, for a set $A$, $\bbone_{A}(x)=1$ if $x\in A$ and 
$\bbone_{A}(x)=0$ otherwise. 
We note that \eqref{e1721975} does not follow 
directly from \eqref{e17514745} or its proof,
but it can be shown again using the Stein-Chen method, 
see \cite[page 80]{MR1163825}. However, more can be said:

\begin{remark}\label{r2230585}
Let the matrix $(p_{j,r})\in[0,1]^{\set{n}\times \set{n}}$ 
have (weakly) decreasing columns, that is 
$p_{j,r}\geq p_{j+1,r}$ for all 
$j\in\set{n-1}$ and $r\in\set{n}$. Then the following is true: 
\begin{enumerate}

\item\label{r2230585.a}
The probability generating function of $P^{S_n}$, namely 
$\psi_{S_n}(z)=\sum_{k=0}^nP(S_n=k)z^k$ for $z\in\CC$, 
has only real roots.

\item\label{r2230585.b}  
The distribution $P^{S_n}$ is a Bernoulli convolution,
that is, it is the distribution of the sum of $n$ independent Bernoulli 
random variables. 

\item\label{r2230585.c} Inequality \eqref{e1721975} holds and 
\begin{align} \label{e23570845}
\frac{1}{14}\min\Bigl\{1,\frac{1}{\lambda}\Bigr\}(\lambda-\Var S_n)
\leq \dtv{P^{S_n}}{\Po(\lambda)}.
\end{align}

\end{enumerate}
\end{remark}
\begin{proof}
Since $\psi_{S_n}(z)$ is equal to the permanent of the matrix 
$(1+p_{j,r}(z-1))\in\CC^{\set{n}\times \set{n}}$ divided by $n!$, 
part (\ref{r2230585.a}) directly follows from the monotone 
column permanent theorem, see \citet{MR3051161}. 
It is well-known that (\ref{r2230585.a}) implies 
(\ref{r2230585.b}), e.g., see \citet[Proposition 1]{MR1429082}. 
Part (\ref{r2230585.c}) is a consequence of (\ref{r2230585.b}),
\citet[Theorem 1]{MR755837} and \cite[Remark 3.2.2]{MR1163825}.
\end{proof}
It is clear that the statement of Remark 
\ref{r2230585} remains valid, if it is assumed that the matrix
$(p_{j,r})$ has decreasing rows, that is, its transpose has decreasing 
columns. The same holds for increasing columns (or rows).
In particular, if $p_{j,1}=\dots=p_{j,n}$ 
for all $j\in\set{n}$, then \eqref{e1721975} and \eqref{e23570845} hold  
and the distribution of $S_n$ is a Bernoulli convolution with success 
probabilities $p_{1,1},\dots,p_{n,1}$.  


Unfortunately, an inequality of the form
\begin{align*} 
\dtv{P^{S_n}}{\Po(\lambda)}
&\leq C\Bigl(1-\frac{\Var S_n}{\lambda}\Bigr)
\end{align*}
with an absolute constant $C\in(0,\infty)$ cannot generally hold. 
Indeed, if $(p_{j,r})\in[0,1]^{\set{n}\times \set{n}}$ is the identity
matrix, then $\Var S_n=1=\lambda$ and $P^{S_n}\neq\Po(\lambda)$;
e.g., see \cite[Example 4.2.1]{MR1163825}. 
Therefore, to get a generally valid upper bound
of $\dtv{P^{S_n}}{\Po(\lambda)}$, 
one has to enlarge $(1-\frac{\Var S_n}{\lambda})$ somewhat. 
In this context,  the following formula for the variance of $S_n$
is useful: 
\begin{align}
\Var S_n
&=\lambda-\sum_{j=1}^n\overline{p}_{j,\bfcdot}^2-\gamma,
\label{e82165}
\end{align}
where 
\begin{align}
\gamma
&=\frac{1}{2n^2(n-1)}\sum_{(j,k)\in\set{n}_{\neq}^2}
\sum_{(r,s)\in\set{n}_{\neq}^2}(p_{j,r}-p_{j,s})(p_{k,r}-p_{k,s}).
\label{e2863576}
\end{align}
For a proof, see Section \ref{s128646}.
Further, we have
\begin{align}
\Var S_n
&=\lambda-\frac{n}{n-1}
\Bigl(\sum_{j\in\set{n}}\overline{p}_{j,\bfcdot}^2
+\sum_{r\in\set{n}}\overline{p}_{\bfcdot,r}^2
-\frac{1}{n^2}\sum_{(j,r)\in\set{n}^2}p_{j,r}^2
-\frac{\lambda^2}{n}\Bigr) \label{e334226}\\
&=\lambda-\frac{1}{n-1} 
\sum_{j\in\set{n}}\sum_{r\in\set{n}}p_{j,r}
\Bigl(\overline{p}_{j,\bfcdot}+\overline{p}_{\bfcdot,r}
-\frac{p_{j,r}}{n}-\overline{p}\Bigr), \nonumber  
\end{align}
from which it easily follows that   
\begin{align}\label{e82166} 
\begin{gathered} 
\lambda(1-m)
\leq 
\Var S_n
\leq \lambda, \quad\text{ where } \\
m=\frac{n}{n-1}
\max_{(j,r)\in\set{n}^2}\Bigl(\overline{p}_{j,\bfcdot}
+\overline{p}_{\bfcdot,r}-\frac{p_{j,r}}{n}-\overline{p}\Bigr)
\leq \min\Bigl\{\lambda,\frac{2n}{n-1}\Bigr\}.
\end{gathered}
\end{align}
Identity \eqref{e334226} and the second inequality in \eqref{e82166}
were proved in \cite[Proposition 4.1.1]{MR1163825}
in the case $(p_{j,r})\in\{0,1\}^{\set{n}\times \set{n}}$
with $\frac{1}{n^2}\sum_{(j,r)\in\set{n}^2}p_{j,r}^2$
replaced by $\overline{p}$.
The proof in the general case is analogous. 
We note that the first inequality in \eqref{e82166} 
can only  be useful if $m<1$. 

The rest of the paper is structured as follows. Sections \ref{s18756}
and \ref{s18647} contain our main approximation inequalities for the 
total variation, Wasserstein and local distances. 
The proofs are given in sections \ref{s2117654} and \ref{s128646}. 
In what follows, let the assumptions of section \ref{s86256} hold 
if not stated otherwise.  
\section{Results for the total variation distance} \label{s18756}
To state our first result, we need further quantities related to 
$\gamma$. Let
\begin{align*}
\gamma'
&=\frac{2}{n^2(n-1)}\sum_{(j,k)\in\set{n}_{\neq}^2}
\sum_{(r,s)\in\set{n}_{\neq}^2}(p_{j,r}-p_{j,s})_+
(p_{k,s}-p_{k,r})_+,\\
\gamma''
&=\frac{1}{2n^2(n-1)}\sum_{(j,k)\in\set{n}_{\neq}^2}
\sum_{(r,s)\in\set{n}_{\neq}^2}\abs{p_{j,r}-p_{j,s}}
\abs{p_{k,r}-p_{k,s}},\\
\gamma'''
&=\frac{1}{4n^2(n-1)}\sum_{(j,k)\in\set{n}_{\neq}^2}
\sum_{(r,s)\in\set{n}_{\neq}^2}
(\abs{p_{j,r}-p_{j,s}}-\abs{p_{k,r}-p_{k,s}})^2.
\end{align*}
Here and below, we set $x_+=\max\{0,x\}$ for $x\in\RR$. 
Our first result improves on \eqref{e17514745}. 
\begin{theorem} \label{t937587}
We have 
\begin{align}
\dtv{P^{S_n}}{\Po(\lambda)}
&\leq \frac{1-\ee^{-\lambda}}{\lambda}(\lambda-\Var S_n+\gamma')
\label{e862146}\\
&=\frac{1-\ee^{-\lambda}}{\lambda}
\Bigl(\sum_{j=1}^n\overline{p}_{j,\bfcdot}^2+\gamma''\Bigr)
\label{e862147}\\
&=\frac{1-\ee^{-\lambda}}{\lambda}
\Bigl(\frac{1}{n}\sum_{(j,r)\in\set{n}^2}p_{j,r}^2-\gamma'''\Bigr).
\label{e862148}
\end{align}
\end{theorem}

Let us compare the inequality in Theorem \ref{t937587} 
with \eqref{e17514745}. We first note that 
the upper bound in Theorem~\ref{t937587} is always smaller than 
or equal to the right-hand side of \eqref{e17514745}. However, 
both upper bounds are of the same order. This follows from the 
next lemma. 
\begin{lemma}\label{l32864}
Let $A=\lambda-\Var S_n+\gamma'$, 
$B=\frac{n-2}{n}(\lambda-\Var S_n)+\frac{2\lambda^2}{n}$. Then we have 
\begin{align}\label{e98256}
A\leq B\leq \Bigl(3-\frac{2}{n}\Bigr) A. 
\end{align}
\end{lemma}

\begin{remark}
In view of \eqref{e862148}, we see that
the bound in Theorem \ref{t937587} is always bounded by 
$1-\ee^{-\lambda}$.  
\end{remark}
\begin{example} 
The right-hand side in \eqref{e17514745} is not always bounded by $1$ 
and it is possible that in the second inequality in \eqref{e98256} 
equality holds. To show this, 
let $p_{j,r}=p\in[0,1]$ for all $j,r\in\set{n}$. Then $P^{S_n}$ 
is the binomial distribution with parameters $n$ and $p$. Further,
$\lambda=np$ and inequality \eqref{e17514745} states that 
\begin{align}\label{e334867}
\dtv{P^{S_n}}{\Po(np)}
&\leq \Bigl(3-\frac{2}{n}\Bigr)(1-\ee^{-np}) p.
\end{align}
On the other hand, Theorem \ref{t937587} 
gives $\dtv{P^{S_n}}{\Po(np)}\leq (1-\ee^{-np})p$. 
In particular, in the second inequality in \eqref{e98256} 
equality holds. If $p>\frac{1}{3}$, then the 
upper bound in \eqref{e334867} is bigger than $1$ for 
$n$ large enough.
\end{example}

\begin{remark}  
If the matrix $(p_{j,r})$ has decreasing rows, then $\gamma'=0$. 
In this case, the inequalities in Theorem \ref{t937587} and 
\eqref{e1721975} are identical.  
We recall that \eqref{e17514745} does not imply \eqref{e1721975}. 
\end{remark}


\begin{remark}  
The distribution of $S_n$ remains unchanged if we replace the matrix 
$(p_{j,r})$ with its transpose. In the case 
$(p_{j,r})\in\{0,1\}^{\set{n}\times \set{n}}$, both upper bounds in 
\eqref{e17514745} and Theorem \ref{t937587} remain unchanged as well, 
see Lemma \ref{l387256}(\ref{l387256.a}) below.  
However, if we consider the general case 
$(p_{j,r})\in[0,1]^{\set{n}\times \set{n}}$, 
the bound in Theorem \ref{t937587} can indeed change.  
So, in this case, we possibly get two different inequalities, 
the better of which should be used. 
For instance, let 
$n=2$ and 
$(p_{j,r})
=\begin{psmallmatrix}
1 & 1/4 \\
3/4 & 1/2 
\end{psmallmatrix}$. Then we have $\gamma'=0$, 
but the analogous term for the transpose of $(p_{j,r})$ is equal to
\begin{align*}
\frac{2}{n^2(n-1)}\sum_{(j,k)\in\set{n}_{\neq}^2}
\sum_{(r,s)\in\set{n}_{\neq}^2}(p_{j,r}-p_{k,r})_+
(p_{k,s}-p_{j,s})_+ 
&=\frac{1}{16}.
\end{align*}
\end{remark}


Let us now collect some properties of $\gamma'$. 
\pagebreak
\begin{lemma}\label{l387256}\rule{0pt}{0pt}
\begin{enumerate}

\item \label{l387256.a} We have 
\begin{align}\label{e1876185}
\gamma'
\leq \frac{2}{n^2(n-1)}\sum_{(j,k)\in\set{n}_{\neq}^2}
\sum_{(r,s)\in\set{n}_{\neq}^2}p_{j,r}(1-p_{j,s})p_{k,s}(1-p_{k,r}).
\end{align}
Here, equality holds if and only if for all $j\in\set{n}$ and 
$(r,s)\in\set{n}_{\neq}^2$ such that $(p_{j,r},p_{j,s})\in(0,1)^2$, we 
have $p_{k,r}=p_{k,s}\in\{0,1\}$ for all $k\in\set{n}\setminus\{j\}$.
In particular, equality holds if in every row of the matrix $(p_{j,r})$ 
there is at most one entry in $(0,1)$. 

The right-hand side of \eqref{e1876185} does not change 
if we replace the matrix $(p_{j,r})$ with its transpose. 

\item \label{l387256.b}
We have
\begin{align}\label{e11864820}
\gamma'
\leq \min\Big\{\Var S_n
-\frac{1}{n}\sum_{(j,r)\in\set{n}^2}p_{j,r}(1-p_{j,r}),
\frac{2}{n}(\Var S_n-\lambda+\lambda^2)\Big\}.
\end{align}

\item 
If $(p_{j,r})\in\{0,1\}^{\set{n}\times \set{n}}$, then
\begin{align*}
\gamma'
&=\frac{2}{n-1}\sum_{(r,s)\in\set{n}_{\neq}^2}
\Bigl(\overline{p}_{\bfcdot,r}
-\frac{1}{n}\sum_{j\in\set{n}}p_{j,r}p_{j,s}\Bigr)
\Bigl(\overline{p}_{\bfcdot,s}
-\frac{1}{n}\sum_{j\in\set{n}}p_{j,r}p_{j,s}\Bigr).
\end{align*}

\end{enumerate}
\end{lemma}
\begin{remark}\label{r24672345}
We note that \eqref{e862146} together with 
the second entry in the right-hand side of \eqref{e11864820} 
can be used to show \eqref{e17514745} in the general case 
$p_{j,r}\in[0,1]^{\set{n}\times \set{n}}$.
Further, an analogue of 
\eqref{e346689} can be shown: in view of \eqref{e17514745}, 
\eqref{e334226} and \eqref{e82166}, we see that 
\begin{align*}
\frac{3}{2}\frac{1-\ee^{-\lambda}}{\lambda}
\Bigl(\sum_{j=1}^n\overline{p}_{j,\bfcdot}^2
+\sum_{r=1}^n\overline{p}_{\bfcdot,r}^2
-\frac{2}{3n^2}\sum_{(j,r)\in\set{n}^2}p_{j,r}^2\Bigr)
\end{align*}
is larger than or equal to the right-hand side of \eqref{e17514745}. 
\end{remark}
\begin{example}
The left entry $L$ and the right entry $R$ in the minimum term in 
\eqref{e11864820} are not comparable in general. 
\begin{enumerate}

\item If $p_{j,r}=1$ for all $j,r\in\set{n}$, 
then $\lambda=n$ and $\Var S_n=0$, that is $L=0<2(n-1)=R$.

\item On the other hand, if
$(p_{j,r})\in[0,1]^{\set{n}\times \set{n}}$ is the identity matrix, then 
$\Var S_n =1=\lambda$; e.g., see \cite[Example 4.2.1]{MR1163825}.
In this case, we have $R=\frac{2}{n}\leq1=L$. 

\end{enumerate}
\end{example}
\begin{remark}\label{r235246346}
Let us discuss conditions for the smallness of the total variation 
distance between $P^{S_n}$ and $\Po(\lambda)$. 
For this, we consider a triangular scheme, where $n$,
all  $X_{j,r}$ and  $p_{j,r}$,  $\pi$, and in turn  
$\lambda$, $\gamma'$, and $S_n$ depend on a further variable $k\in\NN$,
which we let go to infinity later. 
In order to simplify the notation, we do not indicate 
the dependence on $k$. 
In \citet[Corollary 4.A.1]{MR1163825}, it was shown that,
under the assumptions 
\begin{align}\label{e23578}
(p_{j,r})\in\{0,1\}^{n\times n},\quad 
\lim_{k\to\infty}\frac{\lambda}{n}=0 
\quad \text{and}\quad 
\liminf_{k\to\infty} \lambda>0,
\end{align}
we have 
\begin{align}\label{e37477}
\dtv{P^{S_n}}{\Po(\lambda)}\to0 \quad
\text{if and only if} \quad \frac{\Var S_n}{\lambda}\to1. 
\end{align}
However, since \eqref{e17514745} also holds in the case 
$(p_{j,r})\in[0,1]^{n\times n}$ according to Remark \ref{r24672345}, 
the condition $(p_{j,r})\in\{0,1\}^{n\times n}$ in \eqref{e23578} 
can be dropped. Further, a slight refinement of the proof in 
\cite{MR1163825} shows that the condition 
$\liminf_{k\to\infty} \lambda>0$ in \eqref{e23578} can be dropped as 
well. In fact, sufficiency follows from \eqref{e17514745}. 
Further, Theorem 3.A in \cite{MR1163825} shows that 
if $\dtv{P^{S_n}}{\Po(\lambda)}\to0$, then 
$\min\{1,\lambda\}(1-\frac{\Var S_n}{\lambda})\to 0$. 
The first and the third inequality in \eqref{e82166} imply that 
$\lambda-\Var S_n\leq \lambda^2$, leading to 
$(1-\frac{\Var S_n}{\lambda})^2\leq 
\min\{1,\lambda\}(1-\frac{\Var S_n}{\lambda})\to 0$. 
So, the condition $\liminf_{k\to\infty} \lambda>0$
is not needed here.  
If we use  \eqref{e862146}, instead of \eqref{e17514745}, in the 
sufficiency part, we obtain the following corollary without
using the assumptions in \eqref{e23578}. 
\end{remark}

\begin{corollary}\label{c223865}
Consider the triangular scheme above. 
If $\frac{\gamma'}{\lambda}\to0$, then \eqref{e37477} holds. 
\end{corollary}
\begin{remark}  
In the situation of Corollary \ref{c223865}, 
it follows from \eqref{e11864820} and \eqref{e82166} that 
$\frac{\gamma'}{\lambda}\leq 2\frac{\lambda}{n}$. This implies 
that the assumption $\frac{\gamma'}{\lambda}\to0$ 
is weaker than $\frac{\lambda}{n}\to 0$. In particular, 
if the matrices $(p_{j,r})$ have decreasing rows, we have $\gamma'=0$
for all $k$, so that  \eqref{e37477} holds. 
\end{remark}

In the next theorem, we consider the approximation of $P^{S_n}$ by the  
finite signed measure $Q_2$ concentrated on $\Zpl$ with
\begin{align}
\begin{split}
Q_2(A)
&=\Po(\lambda)(A)-\frac{1}{2}(\lambda-\Var S_n) 
((\dirac_1-\dirac_0)^{*2}*\Po(\lambda))(A) \\
&=\sum_{k\in A} \ee^{-\lambda}\frac{\lambda^k}{k!}
\Bigl(1-\frac{1}{2\lambda^2}(\lambda-\Var S_n )(\lambda^2-2k\lambda
+k(k-1)) \Bigr) \quad \text{ for } A\subseteq\Zpl. 
\end{split}\label{e22876346}
\end{align}
Comparable approximations in the case of independent summands 
were considered by \citet{MR0165555}, \citet{MR0428387}, 
\citet{MR0458544}, \citet{MR755837}, \citet{MR1735783} and others. 
\begin{theorem}\label{t2987658}
Let $n\geq 4$. Then
\begin{align}\label{e28649837}
\dtv{P^{S_n}}{Q_2}
\leq \frac{1-\ee^{-\lambda}}{\lambda} \varepsilon,
\end{align}
where  
\begin{align*}
\varepsilon
&=\min\Bigl\{1,\sqrt{\frac{2}{\lambda\ee}}\Bigr\}
(\varepsilon_1+\varepsilon_2) 
+\frac{1-\ee^{-\lambda}}{\lambda}
\varepsilon_3,\\
\varepsilon_1
&=\frac{2}{n} \sum_{j=1}^n\sum_{r=1}^n\overline{p}_{j,\bfcdot}p_{j,r}
\abs{\lambda-\lambda_{j,r}'},\\
\varepsilon_2
&=\frac{1}{n^2(n-1)}
\sum_{(j,k)\in\set{n}_{\neq}^2}\sum_{(r,s)\in\set{n}_{\neq}^2}
\abs{p_{j,r}-p_{j,s}} \abs{p_{k,r}-p_{k,s}}
\abs{\lambda-\lambda_{j,k,r,s}''}, \\
\varepsilon_3
&= \frac{2n^2}{(n-2)^2}
\Bigl(\frac{n-2}{n-1}\sum_{j=1}^n\overline{p}_{j,\bfcdot}^2
+\sqrt{\frac{n-1}{n-3}}\gamma''\Bigr)^2
\end{align*}
and, for $(j,k),(r,s)\in\set{n}_{\neq}^2$, 
\begin{align*}
\lambda_{j,r}'
&=\frac{1}{n-1}\sum_{u\in\set{n}\setminus\{j\}}
\sum_{v\in\set{n}\setminus\{r\}}p_{u,v}
=\frac{1}{n-1}(n(\lambda-\overline{p}_{j,\bfcdot}
-\overline{p}_{\bfcdot,r})+p_{j,r}),\\
\lambda_{j,k,r,s}''
&=\frac{1}{n-2}\sum_{u\in\set{n}\setminus\{j,k\}}
\sum_{v\in\set{n}\setminus\{r,s\}}p_{u,v}\\
&=\frac{1}{n-2}(n(\lambda-\overline{p}_{j,\bfcdot}
-\overline{p}_{k,\bfcdot}
-\overline{p}_{\bfcdot,r}-\overline{p}_{\bfcdot,s})
+p_{j,r}+p_{k,r}+p_{j,s}+p_{k,s}).
\end{align*}
\end{theorem}

In what follows, let 
\begin{align*}
\varepsilon_0
=\frac{1-\ee^{-\lambda}}{\lambda}(\lambda-\Var S_n+\gamma')
=\frac{1-\ee^{-\lambda}}{\lambda}
\Bigl(\sum_{j=1}^n\overline{p}_{j,\bfcdot}^2+\gamma''\Bigr)
\end{align*}
be the upper bound in Theorem \ref{t937587}. 
\begin{remark} 
Let the assumptions of Theorem \ref{t2987658}
hold. Then 
\begin{gather}
\varepsilon_1
\leq 2
(\overline{p}_{\max,\bfcdot}+\overline{p}_{\bfcdot,\max}) 
\sum_{j=1}^n\overline{p}_{j,\bfcdot}^2\,,  \label{e2275483} \\
\varepsilon_2
\leq  4
(\overline{p}_{\max,\bfcdot}+\overline{p}_{\bfcdot,\max}) \gamma'',
\qquad
\Bigl(\frac{1-\ee^{-\lambda}}{\lambda}\Bigr)^2\varepsilon_3
\leq \frac{2n^2(n-1)}{(n-2)^2(n-3)}\varepsilon_0^2,\label{e2275484}
\end{gather}
where 
$\overline{p}_{\max,\bfcdot}
=\max_{j\in\set{n}}\overline{p}_{j,\bfcdot}$
and 
$\overline{p}_{\bfcdot,\max}
=\max_{r\in\set{n}}\overline{p}_{\bfcdot,r}$. 
These inequalities together with \eqref{e28649837} 
lead to the somewhat crude bound 
\begin{align*} 
\dtv{P^{S_n}}{Q_2}
&\leq  4\varepsilon_0
\Bigl(\min\Big\{1,\sqrt{\frac{2}{\lambda\ee}}\Big\}
(\overline{p}_{\max,\bfcdot}+\overline{p}_{\bfcdot,\max}) 
+\frac{n^2(n-1)}{2(n-2)^2(n-3)}
\varepsilon_0\Bigr),  
\end{align*}
which shows that the right-hand side of \eqref{e28649837} 
is of a better order than $\varepsilon_0$. 
The inequalities above are easily proved by using that, 
for $(j,k),(r,s)\in\set{n}_{\neq}^2$,  
\begin{gather*}
\lambda_{j,r}'-\lambda
=\frac{\lambda_{j,r}'}{n}
-\Bigl(\overline{p}_{\bfcdot,r}+\overline{p}_{j,\bfcdot}
-\frac{p_{j,r}}{n}\Bigr),  \\
\lambda_{j,k,r,s}''-\lambda
=\frac{2}{n}\lambda_{j,k,r,s}''
-\Bigl(\overline{p}_{j,\bfcdot}
+\overline{p}_{k,\bfcdot}
+\frac{1}{n}\sum_{u\in\set{n}\setminus\{j,k\}}p_{u,r}
+\frac{1}{n}\sum_{u\in\set{n}\setminus\{j,k\}}p_{u,s}\Bigr)
\end{gather*}
giving
\begin{align*}
\abs{\lambda-\lambda_{j,r}'}
&\leq \max\Big\{\frac{\lambda_{j,r}'}{n},\;
\overline{p}_{\bfcdot,r}+\overline{p}_{j,\bfcdot}
-\frac{p_{j,r}}{n}\Big\}
\leq \overline{p}_{\max,\bfcdot}+\overline{p}_{\bfcdot,\max}
\end{align*}
and
\begin{align*}
\abs{\lambda-\lambda_{j,k,r,s}''}
&\leq \max\Big\{
\frac{2}{n}\lambda_{j,k,r,s}'', \;
\overline{p}_{j,\bfcdot}
+\overline{p}_{k,\bfcdot}
+\frac{1}{n}\sum_{u\in\set{n}\setminus\{j,k\}}p_{u,r}
+\frac{1}{n}\sum_{u\in\set{n}\setminus\{j,k\}}p_{u,s}
\Big\}\\
&\leq2 (\overline{p}_{\max,\bfcdot}+\overline{p}_{\bfcdot,\max}).
\end{align*}
\end{remark}
In what follows, $C$ denotes a positive absolute constant, 
the value of which may change from line to line. Further, 
let $\floor{x}$ for  $x\in\RR$ be the largest integer $\leq x$. 

\begin{corollary}\label{c2975766}
Let the assumptions of Theorem \ref{t2987658} hold. 
Then 
\begin{gather}
\ABS{\dtv{P^{S_n}}{\Po(\lambda)}
-\frac{\lambda-\Var S_n}{\sqrt{2\pii \ee}\,\lambda}}
\leq   C\min\Bigl\{1,\frac{1}{\lambda}\Bigr\}
\Bigl(\frac{\lambda-\Var S_n}{\lambda}+\varepsilon\Bigr),
\label{e297575} \\
\dtv{P^{S_n}}{\Po(\lambda)}
\leq \min\Big\{1,\frac{3}{4\lambda\ee}\Big\}(\lambda-\Var S_n)
+\frac{1-\ee^{-\lambda}}{\lambda}\varepsilon.\label{e2118648}
\end{gather}
The constant $\frac{3}{4\ee}$ in \eqref{e2118648} 
is the best possible.  
\end{corollary}
\begin{remark}\label{r286492}%
Let the assumptions of Theorem \ref{t2987658} hold. 
\begin{enumerate}

\item
Let the matrix $(p_{j,r})$ have decreasing rows. 
Then $P^{S_n}$ is a Bernoulli convolution, see
Remark \ref{r2230585}. Let $\theta=\frac{\lambda-\Var S_n}{\lambda}$.
Since $\gamma'=0$, we have 
\begin{align}\label{e2876474}
\varepsilon_0=(1-\ee^{-\lambda})\theta
=\frac{1-\ee^{-\lambda}}{\lambda}
\Bigl(\sum_{j=1}^n\overline{p}_{j,\bfcdot}^2+\gamma''\Bigr)
\quad\text{and}\quad
\theta=\frac{1}{\lambda}
\Bigl(\sum_{j=1}^n\overline{p}_{j,\bfcdot}^2+\gamma''\Bigr). 
\end{align}
Further, \eqref{e2275483}  and \eqref{e2275484} imply that  
\begin{align}\label{e2836565}
\varepsilon_1
&\leq 4\sum_{j=1}^n\overline{p}_{j,\bfcdot}^2, \quad 
\varepsilon_2
\leq  8 \gamma'' ,\quad
\varepsilon_3
\leq C(\lambda\theta)^2,
\quad
\varepsilon
\leq C\lambda\theta\min\Bigl\{1,\frac{1}{\sqrt{\lambda}}
+\theta\Bigr\}
\end{align}
and \eqref{e297575} yields 
\begin{align*}
\ABS{\dtv{P^{S_n}}{\Po(\lambda)}-\frac{\theta}{\sqrt{2\pii \ee}}}
&\leq  C\theta\min\Big\{1,\frac{1}{\sqrt{\lambda}} +\theta\Big\}.
\end{align*}
The latter also follows from \citet[formula (32)]{MR1735783} and is a 
generalization, resp. refinement, of results in 
\citet[Theorem 2]{MR0056861} (see also \citet[page 2]{MR1163825})
and \citet[Theorem 1.2]{MR832029}. 

\item \label{r286492.b}
Inequality \eqref{e2118648} is a refinement of \eqref{e862146}.
Further, the optimality of the constant $\frac{3}{4\ee}$ on the 
right hand side of \eqref{e2118648} can be verified
by using the special example of $p_{j,r}=p$ for all $j,r\in\set{n}$. 
Here, $S_n$ has a binomial distribution with parameters $n$ and $p$.  
Further, we have $\lambda=np$ and $\Var S_n=np(1-p)$. 
In view of \eqref{e2118648} and the definition of $\varepsilon$, we see 
that 
\begin{align}
\dtv{P^{S_n}}{\Po(\lambda)}
&\leq \frac{3}{4\ee}p+Cp^2. \label{e229837}
\end{align}
From \cite[Theorem 2]{MR1841404}, 
it follows that, in the present situation, 
\begin{align*}
\dtv{P^{S_n}}{\Po(\lambda)}\sim \frac{3}{4\ee}p, 
\end{align*}
as $p\to 0$ and $np\to1$, where we use the triangular scheme as in 
Remark \ref{r235246346}. Here  $\sim$ means that the quotient of both 
sides tends to $1$. 
This shows the optimality of the constant 
$\frac{3}{4\ee}$ in \eqref{e229837} and \eqref{e2118648}. 

\end{enumerate}
\end{remark}

We now consider the setting in the general matching problem; see 
\cite[pages 82--83]{MR1163825}. 

\begin{example}  
\begin{enumerate}

\item\label{ex224976.a}%
Let $m\in\set{n}$, $a_1,\dots,a_m,b_1,\dots,b_m\in\set{n}\cup\{0\}$ with 
$n=\sum_{\ell=1}^ma_\ell=\sum_{\ell=1}^mb_\ell$. 
For $\ell\in\{0,\dots,m\}$, let $A_\ell=\sum_{j=1}^\ell a_j$ and  
$B_\ell=\sum_{j=1}^\ell b_j$, where $\sum_{j=1}^0a_j=0$ denotes the 
empty sum. 
Define $p_{j,r}=1$ for all $(j,r)\in\bigcup_{\ell =1}^{m}
((A_{\ell -1},A_\ell ]\cap\set{n})\times ((B_{\ell -1},B_\ell ]\cap\set{n})$ and 
$p_{j,r}=0$ otherwise. Then 
\begin{gather*}
\overline{p}_{j,\bfcdot}
=\frac{1}{n}\sum_{r=1}^np_{j,r}=\frac{b_\ell }{n}
\quad\text{ if } \ell \in\set{m}\text{ and }j\in (A_{\ell -1},A_\ell ]\cap\set{n},
\\
\overline{p}_{\bfcdot,r}
=\frac{1}{n}\sum_{j=1}^np_{j,r}=\frac{a_\ell }{n}
\quad\text{ if } \ell \in\set{m}\text{ and }r\in (B_{\ell -1},B_\ell ]\cap\set{n},\\
\lambda
=\frac{1}{n}\sum_{\ell =1}^ma_\ell b_\ell, \quad
\sum_{j=1}^n\overline{p}_{j,\bfcdot}^2
=\frac{1}{n^2}\sum_{\ell =1}^ma_\ell b_\ell ^2, \quad
\sum_{r=1}^n\overline{p}_{\bfcdot,r}^2
=\frac{1}{n^2}\sum_{\ell =1}^ma_\ell ^2b_\ell, \\
\Var S_n
=\lambda-\frac{1}{n-1}\Bigl(
\frac{1}{n}\sum_{\ell =1}^ma_\ell b_\ell (a_\ell +b_\ell )
-\lambda^2-\lambda\Bigr)
\end{gather*}
and
\begin{align*}
\gamma'
&=\frac{2}{n^2(n-1)}\sum_{\ell\in\set{m}_{\neq}^2}
a_{\ell(1)}b_{\ell(1)}a_{\ell(2)}b_{\ell(2)}.
\end{align*}
Therefore, \eqref{e862146} gives 
\begin{align}
\dtv{P^{S_n}}{\Po(\lambda)}
&=\frac{1-\ee^{-\lambda}}{\lambda}
\Bigl(\frac{1}{n-1}\Bigl(
\frac{1}{n}\sum_{\ell =1}^ma_\ell b_\ell (a_\ell +b_\ell )
-\lambda^2-\lambda\Bigr)
+\gamma'\Bigr)
\label{e2987359}
\end{align}
slightly improving \eqref{e17514745} in this case, which reads 
as follows: 
\begin{align*}
\dtv{P^{S_n}}{\Po(\lambda)}
&\leq 
(1-\ee^{-\lambda})
\Bigl(\frac{n-2}{\lambda n(n-1)}\Bigl(
\frac{1}{n}\sum_{\ell =1}^ma_\ell b_\ell (a_\ell +b_\ell )
-\lambda^2-\lambda\Bigr)+\frac{2\lambda}{n}\Bigr),
\end{align*}
see also \cite[page 83]{MR1163825}.

\item 
Let the assumptions in part (\ref{ex224976.a}) hold and  
$a_1=\dots=a_m=b_1=\dots=b_m=d$. Then
$\overline{p}_{j,\bfcdot}=\frac{d}{n}$ for $j\in\set{n}$, 
$\overline{p}_{\bfcdot,r}=\frac{d}{n}$ for $r\in\set{n}$, 
$md=n$, $\lambda=d$, and
$\Var S_n
=d-\frac{d(d-1)}{n-1}$. 
Further, \eqref{e2987359} reduces to 
\begin{align}
\dtv{P^{S_n}}{\Po(\lambda)}
&\leq 
(1-\ee^{-d})\Bigl(
\frac{3d-1}{n}- 
\frac{(d-1)(2d-1)}{n(n-1)}\Bigr), \label{e1865845}
\end{align}
which coincides with an inequality in \cite[page 82]{MR1163825} 
proved by different arguments. 

In the case $n\geq 4$, it follows from \eqref{e2118648},
\eqref{e2275483} and \eqref{e2275484} that
\begin{align}
\dtv{P^{S_n}}{\Po(\lambda)}
&\leq \frac{3}{4\ee}\frac{d-1}{n-1}
+ 4\varepsilon_0\Bigl(\frac{2^{3/2}\sqrt{d}}{n\sqrt{\ee}}
+\frac{n^2(n-1)}{2(n-2)^2(n-3)}\varepsilon_0\Bigr) \label{e22598698}\\
&\leq\frac{3}{4\ee}\frac{d-1}{n-1}
+C\Bigl(\frac{d}{n}\Bigr)^2,\nonumber
\end{align}
where $\varepsilon_0$ is the upper bound in \eqref{e1865845}. 
We note that \eqref{e22598698} is better than \eqref{e1865845}
if $\frac{d}{n}$ is small. 
In particular, if $d=1$, we obtain a bound of order 
$(\frac{d}{n})^2$. 

Further, \eqref{e297575}, \eqref{e2275483} and \eqref{e2275484} imply 
that 
\begin{align*}
\ABS{\dtv{P^{S_n}}{\Po(\lambda)}
-\frac{d-1}{\sqrt{2\pii \ee}(n-1)}}
&\leq C\Bigl(\frac{d-1}{d(n-1)}+\Big(\frac{d}{n}\Big)^2\Bigr).
\end{align*}
In particular, it follows that 
\begin{align*}
\dtv{P^{S_n}}{\Po(\lambda)}
\sim \frac{d}{\sqrt{2\pii\ee}\,n}\quad\text{ as } \quad
d\to\infty,\quad n\to\infty,\quad\text{and}\quad \frac{d}{n}\to0. 
\end{align*}
\end{enumerate}
\end{example}

\section{Results for Wasserstein and local distances} \label{s18647}
Let the one-point concentration of a $\Zpl$-valued random variable 
$Z$ be defined by
\begin{align}\label{e286259}
c(Z)=\sup_{m\in\Zpl}P(Z=m).
\end{align}
\begin{theorem} \label{t23597}
For $(j,k)\in\set{n}_{\neq}^2$, 
let $S_n^{(j)}=\sum_{i\in\set{n}\setminus\{j\}}X_{i,\pi(i)}$
and $S_n^{(j,k)}=\sum_{i\in\set{n}\setminus\{j,k\}}X_{i,\pi(i)}$.
Set $\eta_1=\max_{j\in \set{n}}c(S_n^{(j)})$ and 
$\eta_2=\max_{(j,k)\in \set{n}_{\neq}^2}c(S_n^{(j,k)})$. 
Then  
\begin{align}
\normw{P^{S_n}-\Po(\lambda)}
&\leq \min\Big\{1,\frac{4}{3}\sqrt{\frac{2}{\lambda\ee}}\Big\}
\Bigl(\sum_{j=1}^n\overline{p}_{j,\bfcdot}^2+\gamma''\Bigr),
\label{e286485} \\
\normloc{P^{S_n}-\Po(\lambda)}
&\leq 2\frac{1-\ee^{-\lambda}}{\lambda}
\Bigl(\eta_1\sum_{j=1}^n\overline{p}_{j,\bfcdot}^2+\eta_2\gamma''\Bigr). 
\label{e286486} 
\end{align}
\end{theorem}
\begin{remark}\label{r2287658} 
Let the assumptions of Theorem \ref{t23597} hold.  
\begin{enumerate}

\item The right-hand side of \eqref{e286485} is equal to 
\begin{align*}
\min\Big\{1,\frac{4}{3}\sqrt{\frac{2}{\lambda\ee}}\Big\}
(\lambda-\Var S_n+\gamma')
=\min\Big\{1,\frac{4}{3}\sqrt{\frac{2}{\lambda\ee}}\Big\}
\Bigl(\frac{1}{n}\sum_{(j,r)\in\set{n}^2}p_{j,r}^2-\gamma'''\Bigr),
\end{align*}
see Theorem \ref{t937587}. A similar statement holds with respect 
to \eqref{e286486} after 
estimating $\eta_1$ and $\eta_2$ by their maximum. 
In comparison to \eqref{e286485},  the bound in Theorem \ref{t937587} 
contains an additional factor $\lambda^{-1/2}$ as expected; for 
instance, see \cite[Theorem 2]{MR1841404} in the case of
Bernoulli convolutions. 

\item \label{r2287658.b}
For practical applications of \eqref{e286486}, it is 
necessary to use an explicit upper bound of $\eta_1$ and $\eta_2$, 
for which we refer the reader to \citet{Roos2023}. 
In many cases, the bounds are of the order $\lambda^{-1/2}$. 
This would lead to an upper bound of 
$\normloc{P^{S_n}-\Po(\lambda)}$ with an additional factor 
$\lambda^{-1/2}$ in comparison with the bound in Theorem \ref{t937587}. 
However, a detailed discussion is omitted due to lack of space. 

\item 
For $(j,k)\in\set{n}_{\neq}^2$, we have 
\begin{align*}  
c(S_n^{(j)})\leq 2c(S_n^{(j,k)}), 
\end{align*}
since
$P(S_n^{(j)}=m)\leq P(S_n^{(j,k)}\in\{m-1,m\})\leq 2c(S_n^{(j,k)})$
for $m\in\Zpl$.
This implies that $\eta_1\leq 2\eta_2$, which can be used to estimate 
the right-hand side of \eqref{e286486}.  

\end{enumerate}

\end{remark}
For a finite set $B$, let $\card{B}$ denote its cardinality. 
\begin{theorem}\label{t24745}
Under the assumptions of Theorem \ref{t2987658}, we have 
\begin{align}\
\normw{P^{S_n}-Q_2}
&\leq \min\Big\{1,\frac{4}{3}\sqrt{\frac{2}{\lambda\ee}}\Big\}
\varepsilon, \label{e1289483}\\
\normloc{P^{S_n}-Q_2}
&\leq 2\Bigl(\frac{1-\ee^{-\lambda}}{\lambda}\Bigr)^2
(\varepsilon_1+\varepsilon_2+\kappa\varepsilon_3), \label{e02837576}
\end{align}
where, for $(j,k),(r,s)\in\set{n}_{\neq}^2$,  
\begin{align*}
(T_{j,r}')^{B}
&:=\sum_{i\in\set{n}\setminus (\{j\}\cup B)}X_{i,\pi_{j,r}(i)}
\text{ for } B\subseteq\set{n}\setminus\{j\},\\
(T_{j,k,r,s}'')^{B}
&:=\sum_{i\in\set{n}\setminus(\{j,k\}\cup B)}
X_{i,\pi_{j,k,r,s}(i)}
\text{ for } B\subseteq\set{n}\setminus\{j,k\},  
\end{align*}
the random variable
$\pi_{j,r}$, resp. $\pi_{j,k,r,s}$, is independent of $X$ and has
uniform distribution on 
$(\set{n}\setminus\{r\})_{\neq}^{\set{n}\setminus\{j\}}$, 
resp. 
$(\set{n}\setminus\{r,s\})_{\neq}^{\set{n}\setminus\{j,k\}}$, 
and 
\begin{align*}
\kappa=\max\Bigl\{
\max_{B\subseteq \set{n}\setminus\{j\}:\,1\leq \card{B}\leq 2}
c((T_{j,r}')^{B}), 
\max_{B\subseteq \set{n}\setminus\{j,k\}:\,1\leq \card{B}\leq 2}
c((T_{j,k,r,s}'')^{B})\Bigr\}. 
\end{align*}
\end{theorem}


\begin{corollary}\label{c3297583}
Let the assumptions of Theorem \ref{t24745} hold. Then 
\begin{gather}
\normw{P^{S_n}-\Po(\lambda)}
\leq \min\Big\{1,\frac{1}{\sqrt{2\lambda \ee}}\Big\}(\lambda-\Var S_n) 
+\min\Bigl\{1,\frac{4}{3}\sqrt{\frac{2}{\lambda \ee}}\Bigr\}\varepsilon,
\label{e21965783}\\
\ABS{\normw{P^{S_n}-\Po(\lambda)}
-\frac{\lambda-\Var S_n}{\sqrt{2\pii\lambda}}}
\leq C\min\Bigl\{1,\frac{1}{\sqrt{\lambda}}\Bigr\}\Bigl(
\frac{\lambda-\Var S_n}{\sqrt{\lambda}}+\varepsilon \Bigr),
\label{e21965784}\\
\normloc{P^{S_n}-\Po(\lambda)}
\leq \min\Bigl\{1,\frac{1}{2}\Bigl(\frac{3}{2\lambda\ee}\Bigr)^{3/2}
\Bigr\}(\lambda-\Var S_n) 
+ 2  \Bigl(\frac{1-\ee^{-\lambda}}{\lambda}\Bigr)^2
(\varepsilon_1+\varepsilon_2
+\kappa\varepsilon_3), \label{e21965785} \\
\ABS{\normloc{P^{S_n}-\Po(\lambda)}
-\frac{\lambda-\Var S_n}{2\sqrt{2\pii}\lambda^{3/2}}}
\leq  C\min\Bigl\{1,\frac{1}{\lambda}\Bigr\}
\Bigl(\frac{\lambda-\Var S_n}{\lambda^{3/2}} 
+\min\Bigl\{1,\frac{1}{\lambda}\Bigr\}
(\varepsilon_1+\varepsilon_2+\kappa\varepsilon_3)\Bigr).
\label{e21965786}
\end{gather}
The constants $\frac{1}{\sqrt{2\ee}}$ and $(\frac{3}{2\ee})^{3/2}$ 
in \eqref{e21965783} and \eqref{e21965785} are the best possible.
\end{corollary}

\begin{remark} 
Let the assumptions of Theorem \ref{t24745} hold. 
\begin{enumerate}

\item 
Let the matrix $(p_{j,r})$ have decreasing rows. 
Then $P^{S_n}$ is a Bernoulli convolution, see
Remark \ref{r2230585}. Let $\theta=\frac{\lambda-\Var S_n}{\lambda}$.
Then \eqref{e21965784}, \eqref{e2876474} and \eqref{e2836565} imply that 
\begin{align}
\ABS{\normw{P^{S_n}-\Po(\lambda)}
-\frac{\theta\sqrt{\lambda}}{\sqrt{2\pii}}}
&\leq C\theta\sqrt{\lambda}
\min\Bigl\{1,\frac{1}{\sqrt{\lambda}}+\theta\Bigr\}.  \label{e286392}
\end{align}
Further, in many cases, we have 
$\kappa\leq C\lambda^{-1/2}$, see \ref{r2287658}(\ref{r2287658.b}).
In this case, \eqref{e21965786}, \eqref{e2876474} and \eqref{e2836565} 
yield 
\begin{align}
\ABS{\normloc{P^{S_n}-\Po(\lambda)}
-\frac{\theta}{2\sqrt{2\pii\lambda}}}
&\leq  C\frac{\theta}{\sqrt{\lambda}}
\min\Bigl\{1,\frac{1}{\sqrt{\lambda}}+\theta\Bigr\}.  \label{e286393}
\end{align}
For \eqref{e286392} and \eqref{e286393} in the present situation, 
see also \citet[formula (32)]{MR1735783}. 

\item 
The optimality of the constants $\frac{1}{\sqrt{2\ee}}$ and 
$(\frac{3}{2\ee})^{3/2}$ in \eqref{e21965783} and \eqref{e21965785}   
can be verified by using the special example of $p_{j,r}=p$ for all 
$j,r\in\set{n}$. This follows from arguments similar to those given in 
Remark \ref{r286492}(\ref{r286492.b}). 
\end{enumerate}

\end{remark}

\section{Proofs of main results} \label{s2117654}
Let $t\in(0,\infty)$ and  $Y$ be a $\Po(t)$ distributed random variable. 
In what follows, we use the classical Stein-Chen approach without 
coupling; cf.\ \cite{MR0428387} or \cite{MR1163825}. 
This is based on the following idea: 
If $f\in \RR^{\Zpl}$, then a function $g:=g_{t,f}\in\RR^{\Zpl}$ 
exists such that $g(0)=0$ and the so-called Stein equation 
\begin{align}\label{e1864576}
f(m)=t g(m+1)-m g(m) \quad\text{for all } m\in\Zpl. 
\end{align}
holds. It turns out that $g$ is unique on $\NN$ and satisfies
\begin{align*}
g(m+1)=\frac{1}{t\po(m,t)}\sum_{j=0}^m\po(j,t)f(j)
\quad\text{ for all }m\in\Zpl,
\end{align*}
where $\po(m,t)=\ee^{-t}\frac{t^m}{m!}$ for $m\in\Zpl$.
If additionally $\EE \abs{f(Y)}<\infty$ and $\EE f(N)=0$, then 
we also have 
\begin{align}\label{e234567}
g(m+1)=-\frac{1}{t\po(m,t)}\sum_{j=m+1}^\infty\po(j,t)f(j)
\quad\text{ for all }m\in\Zpl. 
\end{align}
Now suppose that $W$ is a $\Zpl$-valued random variable, 
the distribution of which has to be approximated by $\Po(t)$. 
Suppose further, that we want to measure the approximation error in 
terms of differences like $\abs{\EE h(W)-\EE h(Y)}$ 
for certain functions $h\in\RR^{\Zpl}$, where 
we assume that the expectations are finite. 
Letting $f=h-\EE h(Y)$ and $g:=g_{t,f}$, we then obtain from the Stein 
equation that 
\begin{align}\label{e2286584}
\abs{\EE h(W)-\EE h(Y)}
=\abs{\EE f(W)}
=\abs{\EE(t g(W+1)-W g(W))},
\end{align}
where the right-hand side has to be further estimated. To achieve this, 
estimates for $g$ are necessary, some of which are given in the 
following lemma. 


\begin{lemma}\label{l237315}
Let $t\in(0,\infty)$ and  $Y$ be a $\Po(t)$ distributed random variable.
\begin{enumerate}

\item \label{l237315.a}%
Let  $A\subseteq \Zpl$, $h=\bbone_A\in\RR^{\Zpl}$, 
$f=h-\Po(t)(A)\in\RR^{\Zpl}$ and $g_{t,A}:=g_{t,f}$. Then
\begin{align}\label{e11635}
\norminfty{g_{t,A}}
\leq \min\Big\{1,\sqrt{\frac{2}{t\ee}}\Big\}, 
\quad
\norminfty{\Delta g_{t,A}}\leq \frac{1-\ee^{-t}}{t}
\leq \min\Big\{1,\frac{1}{t}\Big\}. 
\end{align}

\item \label{l237315.b}%
Let  $h\in \calfw$ and $f=h-\EE h(Y)\in\RR^{\Zpl}$. Then 
\begin{align}\label{e116361}
\norminfty{g_{t,f}} \leq 1, \quad
\norminfty{\Delta g_{t,f}}
\leq \min\Big\{1,\frac{4}{3}\sqrt{\frac{2}{t\ee}}\Big\}. 
\end{align}

\item \label{l237315.c}
Let $a\in\Zpl$, $h=\bbone_{\{a\}}\in\RR^{\Zpl}$, 
$f=h-\po(a,t)\in\RR^{\Zpl}$ and
$g_{t,\{a\}}$ be defined as in (\ref{l237315.a}). Let 
$Z$ be an arbitrary $\Zpl$-valued random variable and
$c(Z)$ be defined as in \eqref{e286259}.
Then 
\begin{align*}
\norminfty{g_{t,\{a\}}}
\leq 2 \frac{1-\ee^{-t}}{t},\quad
\EE\abs{\Delta g_{t,\{a\}}(Z)}
\leq \min\{1,2 c(Z)\} \frac{1-\ee^{-t}}{t}. 
\end{align*}

\end{enumerate}

\end{lemma}
\begin{proof}
For the proof of the second inequality in (\ref{l237315.a}), 
see \citet[Lemma 4(ii)]{MR0706618}. The remaining inequalities in 
(\ref{l237315.a}) and (\ref{l237315.b}) can be found in 
\citet[pages 7, 14, 15]{MR1163825}. 
The inequalities in (\ref{l237315.c}) follow from 
Lemma 9.1.5 in \cite[page 176]{MR1163825}. Here it is used that 
\begin{align} \label{e25705}
\sum_{m=0}^\infty\abs{\Delta g_{t,\{a\}}(m)}
=2\Delta g_{t,\{a\}}(a),
\end{align}
see the proof of Lemma 4(ii) in \cite{MR0706618} and also 
\citet[page 79]{MR1003970}. 
\end{proof}
The following proposition will be useful in the proofs of 
Corollary \ref{c8263895} and Theorems \ref{t2987658} and \ref{t24745}.

\begin{proposition}\label{l298757}
Let $F,G$ be two finite sets with cardinality 
$m=\card{F}=\card{G}\geq2$, $(p_{j,r})\in[0,1]^{F\times G}$, 
$X_{j,r}$ be a $\Be(p_{j,r})$ distributed random variable
for $(j,r)\in F \times G$, $\rho$ be a random variable uniformly 
distributed over $G^F_{\neq}$.
We assume that all the random variables $\rho$, 
$X_{j,r}$, $(j\in F,r\in G)$  are independent. Let 
$W=\sum_{j\in F}X_{j,\rho(j)}$, 
$\overline{p}_{j,\bfcdot}'=\EE X_{j,\rho(j)}
=\frac{1}{m}\sum_{r\in G}p_{j,r}$, $(j\in F)$,
$\mu=\EE W=\sum_{j\in F}\overline{p}_{j,\bfcdot}'$. 
For $j,k\in F$ with $j\neq k$, 
let $W_j=W-X_{j,\rho(j)}$ and 
$W_{j,k}=W-X_{j,\rho(j)}-X_{k,\rho(k)}$.
For an arbitrary function $h\in\RR^{\Zpl}$, we have 
\begin{align}\label{e82654665}
\EE(\mu h(W+1)-Wh(W))=D_1+D_2,
\end{align}
where 
\begin{align*}
D_1&=\sum_{j\in F}\EE(\overline{p}_{j,\bfcdot}'p_{j,\rho(j)}
\Delta h(W_j+1)), 
\\
D_2&=\frac{1}{2m}\sum_{(j,k)\in F_{\neq}^2}
\EE((p_{j,\rho(j)}-p_{j,\rho(k)})(p_{k,\rho(j)}-p_{k,\rho(k)})
\Delta h(W_{j,k}+1)). 
\end{align*}
\end{proposition}
\begin{proof}
Using Fubini's theorem, we get, for $j\in F$,  
\begin{align*}
\EE(X_{j,\rho(j)}h(W))
&=\sum_{\ell\in G^F_{\neq}}
\EE(\bbone_{\{\rho=\ell\}}X_{j,\ell(j)}h(W_j+X_{j,\ell(j)}))\\
&=\EE(p_{j,\rho(j)}h(W_j+1))
\end{align*}
and similarly 
\begin{align*}
\EE h(W+1)
&=\EE(p_{j,\rho(j)}h(W_j+2)+(1-p_{j,\rho(j)})h(W_j+1)).
\end{align*}
Hence 
\begin{align*}
\lefteqn{\EE(\mu h(W+1)-Wh(W))}\\
&=\sum_{j\in F}\EE(\overline{p}_{j,\bfcdot}'h(W+1)-X_{j,\rho(j)}h(W))\\
&=\sum_{j\in F} 
\EE(\overline{p}_{j,\bfcdot}'p_{j,\rho(j)}h(W_j+2)
+\overline{p}_{j,\bfcdot}'(1-p_{j,\rho(j)})h(W_j+1)
-p_{j,\rho(j)}h(W_j+1))\\
&=D_1+D_2',
\end{align*}
where 
\begin{align*}
D_2'&=\sum_{j\in F} 
\EE((\overline{p}_{j,\bfcdot}'-p_{j,\rho(j)})h(W_j+1)).
\end{align*}
We have to show that $D_2'=D_2$. 
Similarly to the above, we obtain for $(j,k)\in F_{\neq}^2$ that
\begin{align*}
\EE((p_{j,\rho(k)}-p_{j,\rho(j)})h(W_j+1))
&=\EE((p_{j,\rho(k)}-p_{j,\rho(j)})
(p_{k,\rho(k)} h(W_{j,k}+2)\\
&\quad{}+(1-p_{k,\rho(k)})h(W_{j,k}+1))).
\end{align*}
Therefore  
\begin{align}
D_2'
&=\sum_{j\in F} 
\EE((\overline{p}_{j,\bfcdot}'-p_{j,\rho(j)})h(W_j+1))\nonumber\\
&=\frac{1}{m}\sum_{j\in F}\sum_{k\in F}
\EE((p_{j,\rho(k)}-p_{j,\rho(j)})h(W_j+1))\nonumber\\
&=\frac{1}{m}\sum_{(j,k)\in F_{\neq}^2}
\EE((p_{j,\rho(k)}-p_{j,\rho(j)})(p_{k,\rho(k)} h(W_{j,k}+2)
+(1-p_{k,\rho(k)}) h(W_{j,k}+1)))\nonumber\\
&=\frac{1}{m}\sum_{(j,k)\in F_{\neq}^2}
\EE((p_{j,\rho(k)}-p_{j,\rho(j)})p_{k,\rho(k)}
\Delta h(W_{j,k}+1)).  \label{e8257665}
\end{align}
The latter equality follows from the observation that,
for $(j,k)\in F_{\neq}^2$, we have 
\begin{align*}
\EE(p_{j,\rho(k)}h(W_{j,k}+1))=\EE(p_{j,\rho(j)}h(W_{j,k}+1)).
\end{align*}
In fact, the left-hand side does not change if we replace $\rho$ 
with the composition $\rho\circ \tau_{j,k}$, where 
$\tau_{j,k}\in F_{\neq}^F$ is the transposition, which interchanges $j$ 
and $k$. Similarly, we obtain 
\begin{align*}
\EE((p_{j,\rho(k)}-p_{j,\rho(j)})p_{k,\rho(k)}\Delta h(W_{j,k}+1))
&=\EE((p_{j,\rho(j)}-p_{j,\rho(k)})p_{k,\rho(j)}\Delta h(W_{j,k}+1)).
\end{align*}
Hence  
\begin{align}
D_2'
&=-\frac{1}{m}\sum_{(j,k)\in F_{\neq}^2}
\EE((p_{j,\rho(k)}-p_{j,\rho(j)})p_{k,\rho(j)}
\Delta h(W_{j,k}+1))\label{e8257666}.
\end{align}
By adding the right-hand sides in \eqref{e8257665} and \eqref{e8257666} 
and dividing by two we obtain $D_2'=D_2$. 
This proves the assertion. 
\end{proof}

\begin{remark}  
The combination of (2.10) and (2.11) in \citet{MR385977}  
leads to an identity, which is similar but not identical to 
\eqref{e82654665}. It may be possible to get our results by 
using that identity. However, we prefer \eqref{e82654665},
since it does not require additional notation like that in 
\cite[(2.3)--(2.9)]{MR385977}. 
\end{remark}

\begin{corollary} \label{c8263895}
Let the assumptions of Proposition \ref{l298757} hold.
For a $\Zpl$-valued random variable $Z$, let 
$c(Z)$ be defined as in \eqref{e286259}. Set 
\begin{align*}
\alpha=\sum_{j\in F}(\overline{p}_{j,\bfcdot}')^2,\quad
\beta=\frac{1}{2m^2(m-1)}
\sum_{(j,k)\in F_{\neq}^2}\sum_{(r,s)\in G_{\neq}^2}
\abs{p_{j,r}-p_{j,s}} \abs{p_{k,r}-p_{k,s}}.
\end{align*}
Then, for $t\in(0,\infty)$,
\begin{align}
\dtv{P^{W}}{\Po(t)}
&\leq \abs{t-\mu}
\min\Big\{1,\sqrt{\frac{2}{t\ee}}\Big\}
+\frac{1-\ee^{-t}}{t}(\alpha+\beta), \nonumber  \\
\normw{P^{W}-\Po(t)}
&\leq \abs{t-\mu}
+\min\Big\{1,\frac{4}{3}\sqrt{\frac{2}{t\ee}}\Big\}(\alpha+\beta), 
\nonumber \\
\normloc{P^{W}-\Po(t)}
&\leq 2\frac{1-\ee^{-t}}{t}
\Bigl(\abs{t-\mu}+\Bigl(\max_{j\in F}c(W_j)\Bigr)\alpha
+\Bigl(\max_{(j,k)\in F_{\neq}^2}c(W_{j,k})\Bigr)\beta\Bigr). 
\label{e3862658}
\end{align}

\end{corollary}
\begin{proof}
Let $t\in(0,\infty)$ and  $Y$ be a $\Po(t)$ distributed random variable.
Let $h\in\RR^{\Zpl}$ such that $h=\bbone_{A}$ for $A\subseteq \Zpl$
or $h\in\calfw$ or $h=\bbone_{\{a\}}$ for $a\in\Zpl$. 
Let $f=h-\EE h(Y)$ and $g:=g_{t, f}$. 
By using \eqref{e2286584}, Proposition \ref{l298757} and \eqref{e11635}, 
we obtain 
\begin{align*}
\abs{\EE h(W)-\EE h(Y)}
&=\abs{\EE(tg(W+1)-Wg(W))}\\
&=\abs{(t-\mu) \EE g(W+1)}+\abs{\EE(\mu g(W+1)-Wg(W))},
\end{align*}
where 
\begin{align*}
\lefteqn{
\abs{\EE(\mu g(W+1)-Wg(W))}}\\
&\leq 
\sum_{j\in F}\EE (\overline{p}_{j,\bfcdot}'p_{j,\rho(j)}
\abs{\Delta g(W_j+1)})\\
&\quad{}+\frac{1}{2m}\sum_{(j,k)\in F_{\neq}^2}
\EE\abs{(p_{j,\rho(j)}-p_{j,\rho(k)})(p_{k,\rho(j)}-p_{k,\rho(k)})
\Delta g(W_{j,k}+1)}\\
&\leq 
\Bigl(\max_{j\in F}\EE \abs{\Delta g(W_j+1)}\Bigr)
\frac{1}{m!}
\sum_{j\in F}\sum_{\ell\in G^F_{\neq}}
\overline{p}_{j,\bfcdot}'p_{j,\ell(j)} \\
&{}\quad+\Bigl(\max_{(j,k)\in F_{\neq}^2}\EE\abs{\Delta g(W_{j,k}+1)}
\Bigr)\frac{1}{2m\, m!}
\sum_{(j,k)\in F_{\neq}^2}\sum_{\ell\in G^F_{\neq}}
\abs{p_{j,\ell(j)}-p_{j,\ell(k)}} \abs{p_{k,\ell(j)}-p_{k,\ell(k)}}\\
&=\Bigl(\max_{j\in F}\EE \abs{\Delta g(W_j+1)}\Bigr)\alpha
+\Bigl(\max_{(j,k)\in F_{\neq}^2}\EE\abs{\Delta g(W_{j,k}+1)}\Bigr)
\beta. 
\end{align*}
The proof is easily completed by using Lemma \ref{l237315}. 
\end{proof}


\begin{lemma}  
We have
\begin{align}\label{e8164}
\gamma+\gamma'=\gamma'',\qquad 
\gamma''+\gamma'''
&=\frac{1}{n}\sum_{(j,r)\in\set{n}^2}p_{j,r}^2
-\sum_{j\in\set{n}}\overline{p}_{j,\bfcdot}^2 \,.
\end{align}
\end{lemma}
\begin{proof}
For $a,b\in\RR$, we have 
$\abs{ab}-ab
=2(a_+(-b)_++(-a)_+b_+)$, and therefore 
\begin{align*}
\gamma''-\gamma
&=\frac{1}{2n^2(n-1)}\sum_{(j,k)\in\set{n}_{\neq}^2}
\sum_{(r,s)\in\set{n}_{\neq}^2}
(\abs{p_{j,r}-p_{j,s}}\abs{p_{k,r}-p_{k,s}}
-(p_{j,r}-p_{j,s})(p_{k,r}-p_{k,s}))\\
&=\frac{1}{n^2(n-1)}\sum_{(j,k)\in\set{n}_{\neq}^2}
\sum_{(r,s)\in\set{n}_{\neq}^2}
((p_{j,r}-p_{j,s})_+(p_{k,s}-p_{k,r})_+ 
+(p_{j,s}-p_{j,r})_+(p_{k,r}-p_{k,s})_+)\\
&=\frac{2}{n^2(n-1)}\sum_{(j,k)\in\set{n}_{\neq}^2}
\sum_{(r,s)\in\set{n}_{\neq}^2}
(p_{j,r}-p_{j,s})_+(p_{k,s}-p_{k,r})_+
=\gamma'.
\end{align*}
Further, 
\begin{align*}
\gamma''+\gamma'''
&=\frac{1}{4n^2(n-1)}\sum_{(j,k)\in\set{n}_{\neq}^2}
\sum_{(r,s)\in\set{n}_{\neq}^2}
((p_{j,r}-p_{j,s})^2+(p_{k,r}-p_{k,s})^2)\\
&=\frac{1}{n}\sum_{(j,r)\in\set{n}^2}p_{j,r}^2
-\sum_{j\in\set{n}}\overline{p}_{j,\bfcdot}^2 \,.\qedhere
\end{align*}
\end{proof}

\begin{proof}[Proof of Theorem \ref{t937587}]
Using Corollary \ref{c8263895} with 
$F=G=\set{n}$, $m=n$, $\rho=\pi$, $W=S_n$, 
$\overline{p}_{j,\bfcdot}'=\overline{p}_{j,\bfcdot}$, 
and $\mu=t=\lambda$, we obtain
\begin{align*}
\dtv{P^{S_n}}{\Po(\lambda)}
&\leq 
\frac{1-\ee^{-\lambda}}{\lambda}\Bigl(
\sum_{j=1}^n\overline{p}_{j,\bfcdot}^2+\gamma''\Bigr).
\end{align*}
The proof is easily completed by using \eqref{e82165} and 
\eqref{e8164}. 
\end{proof}
The proof of Theorem \ref{t23597} is analogous and therefore omitted. 
The following lemma is needed in the proof of Theorems \ref{t2987658} 
and \ref{t24745}.

\begin{lemma}  
Let $t\in(0,\infty)$, $Y$ be a $\Po(t)$ distributed random variable
and $h\in\RR^{\Zpl}$ such that $\EE \abs{h(Y+3)}<\infty$. Set 
$f=h-\EE h(Y)$ and  $g=g_{t,f}$. 
Then
\begin{align}\label{e2864643}
\int h \,\dd ((\dirac_1-\dirac_0)^{*2}*\Po(t))
=-2\EE(\Delta g(Y+1))
\end{align}
is finite. 
\end{lemma}
\begin{proof}
If $h_0:\,\Zpl\longrightarrow \RR$ is an arbitrary 
function, then $\EE\abs{Yh_0(Y)}<\infty$ if and only if 
$\EE \abs{h_0(Y+1)}<\infty$. In this case, we have that 
$\EE (t h_0(Y+1))=\EE(Yh_0(Y))$.
This implies that 
$\EE \abs{h(Y+j)}<\infty$ for $j\in\{0,1,2,3\}$. 
Since 
\begin{align*}
(\dirac_1-\dirac_0)^{*2}*\Po(t)
=(\dirac_2-2\dirac_1+\dirac_0)*\Po(t)
=P^{Y}-2P^{Y+1}+P^{Y+2}, 
\end{align*}
the left-hand side of \eqref{e2864643} is finite. 
Using \eqref{e1864576},  we obtain 
\begin{align*}
\lefteqn{\int h \,\dd ((\dirac_1-\dirac_0)^{*2}*\Po(t))}\\
&=\EE(h(Y)-2h(Y+1)+h(Y+2))\\
&=\EE(-2f(Y+1)+f(Y+2))\\
&=\EE(-2(t g(Y+2)-(Y+1)g(Y+1))+t g(Y+3)-(Y+2)g(Y+2)). 
\end{align*}
Since 
$\EE f(Y)
=0$, we obtain from \eqref{e234567} that, for $j\in\{0,1,2\}$,
\begin{align*}
\EE\abs{g(Y+j+1)}
&\leq \sum_{m=0}^\infty  
\frac{\po(m,t)}{t\po(m+j,t)}\sum_{k=m+j+1}^\infty\po(k,t)\abs{f(k)}
&= \frac{1}{(j+1)}\EE\abs{f(Y+j+1)} 
<\infty.
\end{align*}
Hence $\EE(tg(Y+j+1))=\EE(Yg(Y+j))$ for $j\in\{0,1,2\}$, 
which implies that
\begin{align*}
\lefteqn{\int h \,\dd ((\dirac_1-\dirac_0)^{*2}*\Po(t))}\\
&=\EE(-2(Y g(Y+1)-(Y+1)g(Y+1))+Y g(Y+2)-(Y+2)g(Y+2))\\
&=-2\EE(\Delta g(Y+1)).\qedhere
\end{align*}
\end{proof}

\begin{proof}[Proof of Theorems \ref{t2987658} and \ref{t24745}]
For $\ell\in \set{n}^n_{\neq}$ and a real valued random variable $Z$, 
we write $\EE_\ell Z=\EE_\ell(Z)=\EE(\bbone_{\{\pi=\ell\}}Z)$ 
whenever this exists. Let $Y$ be a $\Po(\lambda)$ distributed random 
variable independent of $\pi$ and $X$  
and let $h=\bbone_{A}\in\RR^{\Zpl}$ for a set 
$A\subseteq \Zpl$ or $h\in\calfw$. 
Let $f=h-\EE h(Y)$ and $g=g_{\lambda,f}$. 
Using \eqref{e1864576} and Proposition \ref{l298757} 
with $F=G=\set{n}$, $\rho=\pi$, $W=S_n$, and $\mu=\lambda$, we obtain 
\begin{align}
\lefteqn{\EE h(S_n)-\EE h(Y)
=\EE f(S_n) 
=\EE(\lambda g(S_n+1)-S_ng(S_n))}\nonumber\\
&=\sum_{j=1}^n\EE(\overline{p}_{j,\bfcdot}p_{j,\pi(j)}
\Delta g(T_j+1)) \nonumber \\
&\quad{}+\frac{1}{2n}\sum_{(j,k)\in \set{n}_{\neq}^2}
\EE((p_{j,\pi(j)}-p_{j,\pi(k)})(p_{k,\pi(j)}-p_{k,\pi(k)})
\Delta g(T_{j,k}+1)) \nonumber\\
&=\sum_{j=1}^n\sum_{r=1}^n\overline{p}_{j,\bfcdot}p_{j,r}
\sum_{\ell\in\set{n}_{\neq}^n:\,\ell(j)=r}
\EE_\ell\Delta g(T_j+1)\nonumber\\
&\quad{}+\frac{1}{2n}\sum_{(j,k)\in\set{n}_{\neq}^2}
\sum_{(r,s)\in\set{n}_{\neq}^2}
(p_{j,r}-p_{j,s})(p_{k,r}-p_{k,s})
\sum_{\newatop{\ell\in\set{n}_{\neq}^n:}{\,\ell(j)=r,\ell(k)=s}}
\EE_\ell \Delta g(T_{j,k}+1), \label{e81645}
\end{align}
where $T_j=S_n-X_{j,\pi(j)}$ and 
$T_{j,k}=S_n-X_{j,\pi(j)}-X_{k,\pi(k)}$ 
for $(j,k)\in\set{n}_{\neq}^2$. 
Combining \eqref{e22876346}, \eqref{e82165}, 
\eqref{e2864643}, \eqref{e81645} and \eqref{e2863576}, we obtain
\begin{align}
\EE h(S_n)-\int h \,\dd Q_2
&=\EE f(S_n)
-\Bigl(\sum_{j=1}^n\overline{p}_{j,\bfcdot}^2+\gamma\Bigr)
\EE (\Delta g(Y+1)) 
=D_1+D_2, \label{e287458}
\end{align}
where
\begin{align*}
D_1
&=\sum_{j=1}^n\sum_{r=1}^n\overline{p}_{j,\bfcdot}p_{j,r}
\sum_{\ell\in\set{n}_{\neq}^n:\,\ell(j)=r}
\EE_\ell (\Delta g(T_j+1)-\Delta g(Y+1))\\
D_2&=\frac{1}{2n}\sum_{(j,k)\in\set{n}_{\neq}^2}
\sum_{(r,s)\in\set{n}_{\neq}^2}
(p_{j,r}-p_{j,s})(p_{k,r}-p_{k,s})
\sum_{\newatop{\ell\in\set{n}_{\neq}^n:}{\,\ell(j)=r,\ell(k)=s}}
\EE_\ell (\Delta g(T_{j,k}+1)-\Delta g(Y+1)).
\end{align*}
We note that, for $j,r\in\set{n}$ and $m\in\Zpl$, 
\begin{align}
\sum_{\ell\in\set{n}_{\neq}^n:\,\ell(j)=r}P(\pi=\ell,T_j=m)
&=\frac{1}{n!}\sum_{\ell\in\set{n}_{\neq}^n:\,\ell(j)=r}
P\Bigl(\sum_{i\in\set{n}\setminus\{j\}}X_{i,\ell(i)}=m\Bigr) 
\nonumber \\
&=\frac{1}{n}\sum_{\ell\in(\set{n}\setminus\{r\})_{\neq}^{\set{n}
\setminus\{j\}}} P(\pi_{j,r}=\ell,T_{j,r}'=m)
=\frac{1}{n}P(T_{j,r}'=m), \label{e2239875}
\end{align}
where $T_{j,r}':=\sum_{i\in\set{n}\setminus\{j\}}
X_{i,\pi_{j,r}(i)}$ and the random variable 
$\pi_{j,r}$ is independent of $X$ and has uniform distribution on 
$(\set{n}\setminus\{r\})_{\neq}^{\set{n}\setminus\{j\}}$. 
Hence 
\begin{align}
\lefteqn{\ABS{\sum_{\ell\in\set{n}_{\neq}^n:\,\ell(j)=r}
\EE_\ell(\Delta g(T_j+1)-\Delta g(Y+1))}} \nonumber \\
&=\ABS{\sum_{m\in\Zpl}\sum_{\ell\in\set{n}_{\neq}^n:\,\ell(j)=r}
(P(\pi=\ell,T_j=m)-P(\pi=\ell,Y=m)) \Delta g(m+1)} \nonumber \\
&=\ABS{\frac{1}{n}\sum_{m\in\Zpl}
(P(T_{j,r}'=m)-P(Y=m)) \Delta g(m+1) } \nonumber \\
&\leq \frac{2\norminfty{\Delta g}}{n}\dtv{P^{T_{j,r}'}}{\Po(\lambda)}. 
\label{e2349755}
\end{align}
We have $\EE T_{j,r}'=\lambda_{j,r}'$ for $j,r\in\set{n}$. 
Using Corollary \ref{c8263895} with 
$F=\set{n}\setminus\{j\}$, $G=\set{n}\setminus\{r\}$, $m=n-1$, 
$\rho=\pi_{j,r}$, $W=T_{j,r}'$,
$\overline{p}_{u,\bfcdot}'
=\frac{1}{n-1}\sum_{v\in\set{n}\setminus\{r\}}p_{u,v}$ for 
$u\in\set{n}\setminus\{j\}$, 
$\mu=\lambda_{j,r}'$ and $t=\lambda$, we obtain 
\begin{align*}
\dtv{P^{T_{j,r}'}}{\Po(\lambda)}
&\leq \abs{\lambda-\lambda_{j,r}'}
\min\Big\{1,\sqrt{\frac{2}{\lambda\ee}}\Big\}
+\frac{1-\ee^{-\lambda}}{\lambda}
\Bigl(\sum_{u\in \set{n}\setminus\{j\}}(\overline{p}_{u,\bfcdot}')^2\\
&\quad{}
+\frac{1}{2(n-1)^2(n-2)}
\sum_{(u,u')\in (\set{n}\setminus\{j\})_{\neq}^2}
\sum_{(v,v')\in (\set{n}\setminus\{r\})_{\neq}^2}
\abs{p_{u,v}-p_{u,v'}} \abs{p_{u',v}-p_{u',v'}}\Bigr)\\
&\leq \abs{\lambda-\lambda_{j,r}'}
\min\Big\{1,\sqrt{\frac{2}{\lambda\ee}}\Big\}
+\frac{n^2}{(n-1)^2}\frac{1-\ee^{-\lambda}}{\lambda}
\Bigl(\sum_{u\in \set{n}}\overline{p}_{u,\bfcdot}^2+
\frac{n-1}{n-2}\gamma''\Bigr).
\end{align*}
Similarly as in \eqref{e2239875}, we obtain that,   
for $(j,k),(r,s)\in\set{n}_{\neq}^2$ and $m\in\Zpl$, 
\begin{align*}
\sum_{\newatop{\ell\in\set{n}_{\neq}^n:}{\ell(j)=r,\ell(k)=s}}
P(\pi=\ell,T_{j,k}=m)
&=\frac{1}{n(n-1)}P(T_{j,k,r,s}''=m), 
\end{align*}
where $T_{j,k,r,s}'':=\sum_{i\in\set{n}\setminus\{j,k\}}
X_{i,\pi_{j,k,r,s}(i)}$ and the random variable 
$\pi_{j,k,r,s}$ is independent of $X$ and 
has uniform distribution on 
$(\set{n}\setminus\{r,s\})_{\neq}^{\set{n}\setminus\{j,k\}}$.
Similarly as in \eqref{e2349755}, we get 
\begin{align}
\lefteqn{
\ABS{\sum_{\newatop{\ell\in\set{n}_{\neq}^n:}{
\ell(j)=r,\ell(k)=s}}
\EE_\ell ( \Delta g(T_{j,k}+1)-\Delta g(Y+1) )} } \nonumber \\
&=\frac{1}{n(n-1)}\ABS{\sum_{m\in\Zpl}
(P(T_{j,k,r,s}''=m)-P(Y=m)) \Delta g(m+1)} \nonumber \\
&\leq \frac{2 \norminfty{\Delta g}}{n(n-1)}
\dtv{P^{T_{j,k,r,s}''}}{\Po(\lambda)}. \label{e235845}
\end{align}
We note that  
$\lambda_{j,k,r,s}''
=\EE T_{j,k,r,s}''$ 
for $(j,k),(r,s)\in\set{n}_{\neq}^2$.
Using Corollary \ref{c8263895} with 
$F=\set{n}\setminus\{j,k\}$, $G=\set{n}\setminus\{r,s\}$, $m=n-2$, 
$\rho=\pi_{j,k,r,s}$, $W=T_{j,k,r,s}''$,
$\overline{p}_{u,\bfcdot}'
=\frac{1}{n-2}\sum_{v\in\set{n}\setminus\{r,s\}}p_{u,v}$
for $u\in\set{n}\setminus\{j,k\}$, 
$\mu=\lambda_{j,k,r,s}''$ and $t=\lambda$, we obtain 
\begin{align*}
\lefteqn{\dtv{P^{T_{j,k,r,s}''}}{\Po(\lambda)}
\leq \abs{\lambda-\lambda_{j,k,r,s}''}
\min\Big\{1,\sqrt{\frac{2}{\lambda\ee}}\Big\}
+\frac{1-\ee^{-\lambda}}{\lambda}
\Bigl(\sum_{u\in \set{n}\setminus\{j,k\}}
(\overline{p}_{u,\bfcdot}')^2}\\
&\quad{}
+\frac{1}{2(n-2)^2(n-3)}
\sum_{(u,u')\in (\set{n}\setminus\{j,k\})_{\neq}^2}
\sum_{(v,v')\in (\set{n}\setminus\{r,s\})_{\neq}^2}
\abs{p_{u,v}-p_{u,v'}} \abs{p_{u',v}-p_{u',v'}}\Bigr)\\
&\leq \abs{\lambda-\lambda_{j,k,r,s}''}
\min\Big\{1,\sqrt{\frac{2}{\lambda\ee}}\Big\}
+\frac{n^2}{(n-2)^2}\frac{1-\ee^{-\lambda}}{\lambda}
\Bigl(\sum_{u\in \set{n}}\overline{p}_{u,\bfcdot}^2
+\frac{n-1}{n-3}\gamma''\Bigr).
\end{align*}
Combining the inequalities above, we see that
\begin{align*}
\abs{D_1}
&\leq \frac{2\norminfty{\Delta g} }{n}
\sum_{j=1}^n\sum_{r=1}^n\overline{p}_{j,\bfcdot}p_{j,r}
\dtv{P^{T_{j,r}'}}{\Po(\lambda)}\\
&\leq 
\norminfty{\Delta g} \min\Big\{1,\sqrt{\frac{2}{\lambda\ee}}\Big\}
\varepsilon_1  
+ 2\norminfty{\Delta g}
\frac{n^2}{(n-1)^2}\frac{1-\ee^{-\lambda}}{\lambda}
\sum_{j=1}^n\overline{p}_{j,\bfcdot}^2
\Bigl(\sum_{u\in \set{n}}\overline{p}_{u,\bfcdot}^2+
\frac{n-1}{n-2}\gamma''\Bigr)
\end{align*}
and 
\begin{align*}
\abs{D_2}
&\leq \frac{\norminfty{\Delta g}}{n^2(n-1)} 
\sum_{(j,k)\in\set{n}_{\neq}^2}\sum_{(r,s)\in\set{n}_{\neq}^2}
\abs{p_{j,r}-p_{j,s}}\abs{p_{k,r}-p_{k,s}}
\dtv{P^{T_{j,k,r,s}''}}{\Po(\lambda)}\\
&\leq 
\norminfty{\Delta g}\varepsilon_2
\min\Big\{1,\sqrt{\frac{2}{\lambda\ee}}\Big\}
+2\norminfty{\Delta g}
\frac{n^2}{(n-2)^2}\frac{1-\ee^{-\lambda}}{\lambda}
\Bigl(\sum_{u\in \set{n}}\overline{p}_{u,\bfcdot}^2
+\frac{n-1}{n-3}\gamma''\Bigr)\gamma''.
\end{align*}
Therefore, using \eqref{e287458},  
\begin{align*}
\ABS{\EE h(S_n)-\int h \,\dd Q_2}
&\leq \abs{D_1}+\abs{D_2}\\
&\leq \norminfty{\Delta g}
\Bigl(\min\Big\{1,\sqrt{\frac{2}{\lambda\ee}}\Big\}
(\varepsilon_1+\varepsilon_2) 
+\frac{1-\ee^{-\lambda}}{\lambda}
\varepsilon_3\Bigr) 
=\norminfty{\Delta g} \varepsilon. 
\end{align*}
In view of  \eqref{e11635} and \eqref{e116361}, 
we see that \eqref{e28649837} and \eqref{e1289483} hold. 

We now give a proof of \eqref{e02837576}. Here, we  
consider the special case that 
$a\in\Zpl$, $h=\bbone_{\{a\}}\in\RR^{\Zpl}$, $f=h-\po(a,\lambda)$ and 
$g=g_{\lambda,\{a\}}$.  
It follows from \eqref{e25705} that, for an arbitrary 
$\Zpl$-valued random variable $Z$,  
\begin{align*}
\ABS{\sum_{m=0}^\infty
(P(Z=m)-P(Y=m)) \Delta g(m+1) }
&\leq \normloc{P^Z-\Po(\lambda)}\sum_{m=0}^\infty  \abs{\Delta g(m)}\\
&\leq 2 \norminfty{\Delta g} \normloc{P^Z-\Po(\lambda)}.
\end{align*}
Under the present assumptions, this leads 
to the following improvements of \eqref{e2349755} 
and \eqref{e235845} for $(j,k),(r,s)\in\set{n}_{\neq}^2$:
\begin{align*}
\ABS{\sum_{\ell\in\set{n}_{\neq}^n:\,\ell(j)=r}
\EE_\ell(\Delta g(T_j+1)-\Delta g(Y+1))}
&\leq\frac{2\norminfty{\Delta g}}{n}
\normloc{P^{T_{j,r}'}-\Po(\lambda)},\\
\ABS{\sum_{\newatop{\ell\in\set{n}_{\neq}^n:}{
\ell(j)=r,\ell(k)=s}}
\EE_\ell ( \Delta g(T_{j,k}+1)-\Delta g(Y+1) )}
&\leq \frac{2\norminfty{\Delta g} }{n(n-1)}
\normloc{P^{T_{j,k,r,s}''}-\Po(\lambda)}.
\end{align*}
Further, from \eqref{e3862658}, we get  
\begin{align*}
\normloc{P^{T_{j,r}'}-\Po(\lambda)}
&\leq 
2 \frac{1-\ee^{-\lambda}}{\lambda}
\Bigl(\abs{\lambda-\lambda_{j,r}'}
+\frac{\kappa n^2}{(n-1)^2}
\Bigl(\sum_{u\in \set{n}}\overline{p}_{u,\bfcdot}^2
+\frac{n-1}{n-2}\gamma''\Bigr)\Bigr),
\end{align*}
and
\begin{align*}
\normloc{P^{T_{j,k,r,s}''}-\Po(\lambda)}
&\leq 2\frac{1-\ee^{-\lambda}}{\lambda}
\Bigl(\abs{\lambda-\lambda_{j,k,r,s}''}
+\frac{\kappa n^2}{(n-2)^2}
\Bigl(\sum_{u\in \set{n}}\overline{p}_{u,\bfcdot}^2
+\frac{n-1}{n-3}\gamma'' \Bigr)\Bigr) .
\end{align*}
In view of \eqref{e287458}, we see that 
\begin{align*}
\EE h(S_n)-\int h \,\dd Q_2
&=D_1+D_2,
\end{align*}
where 
\begin{align*}
\abs{D_1}
&\leq 
2\norminfty{\Delta g}\frac{1-\ee^{-\lambda}}{\lambda}
\varepsilon_1
+4\kappa \norminfty{\Delta g}
\frac{ n^2}{(n-1)^2}
\frac{1-\ee^{-\lambda}}{\lambda}
\sum_{j=1}^n\overline{p}_{j,\bfcdot}^2 
\Bigl(\sum_{u\in \set{n}}\overline{p}_{u,\bfcdot}^2
+\frac{n-1}{n-2}\gamma''\Bigr),  
\end{align*}
and 
\begin{align*}
\abs{D_2}
&\leq 
2 \norminfty{\Delta g} \frac{1-\ee^{-\lambda}}{\lambda}
\varepsilon_2
+ 4\kappa \norminfty{\Delta g}
\frac{n^2}{(n-2)^2} 
\frac{1-\ee^{-\lambda}}{\lambda}
\Bigl( \sum_{u\in \set{n}}\overline{p}_{u,\bfcdot}^2
+\frac{n-1}{n-3}\gamma'' \Bigr) \gamma''.
\end{align*}
Therefore
\begin{align*}
\ABS{\EE h(S_n)-\int h \,\dd Q_2}
\leq \abs{D_1}+\abs{D_2}  
&\leq 2  \norminfty{\Delta g}\frac{1-\ee^{-\lambda}}{\lambda}
(\varepsilon_1+\varepsilon_2
+\kappa\varepsilon_3),
\end{align*}
which together with \eqref{e11635} implies \eqref{e02837576}. 
\end{proof}


\begin{lemma}  \label{l3297576}
Let $t\in(0,\infty)$. Then 
\begin{gather}
\normtv{(\dirac_1-\dirac_0)^{*2}*\Po(t)}
\leq \min\Big\{4,\frac{3}{t\ee}\Big\},
\label{e81265457}\\
\ABS{\normtv{(\dirac_1-\dirac_0)^{*2}*\Po(t)}
- \frac{4}{t\sqrt{2\pii\ee}}}
\leq \frac{C}{t}\min\Big\{1,\frac{1}{t}\Big\},
\label{e81265458} \\
\normw{(\dirac_1-\dirac_0)^{*2}*\Po(t)}
=\normtv{(\dirac_1-\dirac_0)*\Po(t)}
\leq \min\Bigl\{2,\sqrt{\frac{2}{t\ee}}\Bigr\},
\label{e81265459}\\
\ABS{
\normw{(\dirac_1-\dirac_0)^{*2}*\Po(t)}
-\sqrt{\frac{2}{\pii t}}}
\leq \frac{C}{\sqrt{t}}\min\Bigl\{1,\frac{1}{\sqrt{t}}\Bigr\},
\label{e812654591}\\
\normloc{(\dirac_1-\dirac_0)^{*2}*\Po(t)}
\leq \min\Bigl\{2,\Bigl(\frac{3}{2t\ee}\Bigr)^{3/2}\Bigr\},
\label{e81265460}\\
\ABS{\normloc{(\dirac_1-\dirac_0)^{*2}*\Po(t)}
-\frac{1}{\sqrt{2\pii} t^{3/2}}}
\leq \frac{C}{t^{3/2}}\min\Bigl\{1,\frac{1}{t}\Bigr\}.
\label{e812654601}
\end{gather}
The constants $\frac{3}{\ee}$, $\sqrt{\frac{2}{\ee}}$ and 
$(\frac{3}{2\ee})^{3/2}$ on the right-hand sides of 
\eqref{e81265457}, \eqref{e81265459} and \eqref{e81265460} are 
the best possible. 
\end{lemma}
\begin{proof}
For \eqref{e81265457} and \eqref{e81265460} and the optimality 
of the constants $\frac{3}{\ee}$ and $(\frac{3}{2\ee})^{3/2}$, see 
\cite[Lemma 3]{MR1841404}. 
Inequality \eqref{e81265458} was proved in 
\cite[Lemma 5]{MR1965122}; the proof of \eqref{e812654601} 
is analogous using formula (45) in \cite{MR1967760}. 
The equality in \eqref{e81265459} 
follows from \eqref{e2985839}; 
the inequality and the optimality  of the constant
$\sqrt{\frac{2}{\ee}}$ are contained in
\cite[formula (3.8)]{MR0996692}.
Inequality \eqref{e812654591} follows from the more general 
Proposition 4 in \cite[]{MR1735783}. 
\end{proof}
\begin{remark}  
We do not know, if $\frac{1}{\sqrt{t}}$ in the minimum term 
in \eqref{e812654591} can by replaced by $\frac{1}{t}$. 
The norm terms in Lemma \ref{l3297576} can be evaluated using the 
zeros of some Charlier polynomials. 
The details are omitted here; more general identities can be found in 
\cite[Corollaries 1,2]{MR1735783}. 
\end{remark}

\begin{proof}[Proof of Corollary \ref{c2975766}]
Using \eqref{e22876346}, 
Theorem \ref{t2987658}, \eqref{e81265457} and \eqref{e81265458}, we 
obtain 
\begin{align*}
\dtv{P^{S_n}}{\Po(\lambda)}
&=\frac{1}{2}\normtv{P^{S_n}-\Po(\lambda)}\\ 
&\leq\frac{1}{4}(\lambda-\Var S_n) 
\normtv{(\dirac_1-\dirac_0)^{*2}*\Po(\lambda)}+\dtv{P^{S_n}}{Q_2}\\
&\leq \min\Big\{1,\frac{3}{4\lambda \ee}\Big\}(\lambda-\Var S_n) 
+\frac{1-\ee^{-\lambda}}{\lambda}\varepsilon
\end{align*}
and
\begin{align*}
\lefteqn{\ABS{\dtv{P^{S_n}}{\Po(\lambda)}
-\frac{\lambda-\Var S_n}{\sqrt{2\pii \ee} \,\lambda}}
=\ABS{\frac{1}{2}\normtv{P^{S_n}-\Po(\lambda)}
-\frac{\lambda-\Var S_n}{\sqrt{2\pii \ee} \,\lambda}}}\\
&\leq\frac{1}{2}\Normtv{P^{S_n}-\Po(\lambda)
+\frac{1}{2}(\lambda-\Var S_n) (\dirac_1-\dirac_0)^{*2}*\Po(\lambda)}\\
&\quad{}+\frac{1}{4}(\lambda-\Var S_n) 
\ABS{\normtv{(\dirac_1-\dirac_0)^{*2}*\Po(\lambda)}
-\frac{4}{\sqrt{2\pii \ee}\,\lambda}}\\
&\leq\frac{1}{4}(\lambda-\Var S_n) 
\ABS{\normtv{(\dirac_1-\dirac_0)^{*2}*\Po(\lambda)}
-\frac{4}{\lambda\sqrt{2\pii \ee}}}
+\frac{1}{2}\normtv{P^{S_n}-Q_2}\\
&\leq \frac{C}{\lambda}
\min\Bigl\{1,\frac{1}{\lambda}\Bigr\}(\lambda-\Var S_n) 
+\frac{1-\ee^{-\lambda}}{\lambda}\varepsilon\\
&\leq C\min\Bigl\{1,\frac{1}{\lambda}\Bigr\}
 \Bigl(\frac{\lambda-\Var S_n}{\lambda}+\varepsilon\Bigr). \qedhere
\end{align*}
\end{proof}
The proof of Corollary \ref{c3297583} is analogous and therefore 
omitted. 

\section{Remaining proofs} \label{s128646}
\begin{proof}[Proof of \eqref{e82165}]
For $(j,k)\in\set{n}_{\neq}^2$, we have
\begin{gather*}
\Var X_{j,\pi(j)} 
=\overline{p}_{j,\bfcdot}(1-\overline{p}_{j,\bfcdot}),\quad
\EE(X_{j,\pi(j)}X_{k,\pi(k)})
=\frac{1}{n(n-1)}\sum_{(r,s)\in\set{n}_{\neq}^2}p_{j,r}p_{k,s},\\
\Cov(X_{j,\pi(j)},X_{k,\pi(k)})
=\frac{1}{n-1}\Bigl(\overline{p}_{j,\bfcdot}\overline{p}_{k,\bfcdot}
-\frac{1}{n}\sum_{r\in\set{n}}p_{j,r}p_{k,r}\Bigr).
\end{gather*}
Therefore
\begin{align*}
\sum_{(j,k)\in\set{n}_{\neq}^2}\Cov(X_{j,\pi(j)},X_{k,\pi(k)})
&=\frac{1}{n^2(n-1)}\sum_{(j,k)\in\set{n}_{\neq}^2}
\sum_{(r,s)\in\set{n}_{\neq}^2}p_{j,r}(p_{k,s}-p_{k,r})\\
&=-\frac{1}{2n^2(n-1)}\sum_{(j,k)\in\set{n}_{\neq}^2}
\sum_{(r,s)\in\set{n}_{\neq}^2}(p_{j,r}-p_{j,s})(p_{k,r}-p_{k,s})
=-\gamma.
\end{align*}
Hence
\begin{align*}
\Var S_n
&=\sum_{j=1}^n\Var X_{j,\pi(j)}
+\sum_{(j,k)\in\set{n}_{\neq}^2}
\Cov(X_{j,\pi(j)},X_{k,\pi(k)})
=\lambda-\sum_{j=1}^n \overline{p}_{j,\bfcdot}^2-\gamma.\qedhere
\end{align*}
\end{proof}
\begin{proof}[Proof of Lemma \ref{l387256}]{}
\begin{enumerate}[(a)]

\item It is easily shown that, for $a,b\in[0,1]$, we have 
$(a-b)_+\leq a(1-b)$ with equality if and only if 
$(a,b)\notin(0,1)^2$. 
In particular, this implies \eqref{e1876185}. 
To prove the second assertion, we note that, for $a,b,c,d\in[0,1]$, 
we have 
$(a-b)_+(c-d)_+\leq a(1-b)c(1-d)$ with equality if and only if one of 
the following conditions is true 
\begin{enumerate}[(i)]

\item $(a,b)\notin(0,1)^2$ and $(c,d)\notin(0,1)^2$,

\item $(a,b)\in(0,1)^2$ and ($c=0$ or $d=1$), 

\item $(c,d)\in(0,1)^2$ and ($a=0$ or $b=1$). 

\end{enumerate}
The proof of the second assertion is now easily completed:
Let us first show necessity and suppose that equality holds. 
For all $(j,k),(r,s)\in\set{n}_{\neq}^2$, we then have 
\begin{align*}
(p_{j,r}-p_{j,s})_+(p_{k,s}-p_{k,r})_+
&=p_{j,r}(1-p_{j,s})p_{k,s}(1-p_{k,r})
\mbox{ and } \\
(p_{j,s}-p_{j,r})_+(p_{k,r}-p_{k,s})_+
&=p_{j,s}(1-p_{j,r})p_{k,r}(1-p_{k,s}). 
\end{align*}
Now, if $j\in\set{n}$ and $(r,s)\in\set{n}_{\neq}^2$ such that 
$(p_{j,r},p_{j,s})\in(0,1)^2$, then 
$(p_{j,r}-p_{j,s})_+<p_{j,r}(1-p_{j,s})$ and 
$(p_{j,s}-p_{j,r})_+<p_{j,s}(1-p_{j,r})$.
This together with the equalities above implies that, for all 
$k\in\set{n}\setminus\{j\}$, we have 
$p_{k,s}(1-p_{k,r})=0$ and $p_{k,r}(1-p_{k,s})=0$
that is ($p_{k,r}=1$ or $p_{k,s}=0$)
and ($p_{k,r}=0$ or $p_{k,s}=1$), which is equivalent to 
$p_{k,r}=p_{k,s}\in\{0,1\}$. This proves necessity. 
Sufficiency is shown similarly.  

\item We have
\begin{align*}
\gamma'
&\leq  \frac{2}{n^2(n-1)}\sum_{(j,k)\in\set{n}_{\neq}^2}
\sum_{(r,s)\in\set{n}_{\neq}^2}p_{j,r}p_{k,s}\\
&=\frac{2}{n-1}\Bigl(\lambda^2
-\sum_{j\in\set{n}}\overline{p}_{j,\bfcdot}^2
-\sum_{r\in\set{n}}\overline{p}_{\bfcdot,r}^2
+\frac{1}{n^2}\sum_{(j,r)\in\set{n}^2}p_{j,r}^2\Bigr)\\
&=\frac{2}{n}(\Var S_n -\lambda+\lambda^2). 
\end{align*}
On the other hand, using \eqref{e8164}, we get
\begin{align*}
\gamma'=\gamma''-\gamma
&\leq \frac{1}{n}\sum_{(j,r)\in\set{n}^2}p_{j,r}^2
-\sum_{j\in\set{n}}\overline{p}_{j,\bfcdot}^2-\gamma
=\frac{1}{n}\sum_{(j,r)\in\set{n}^2}p_{j,r}^2
+\Var S_n-\lambda\\
&=\Var S_n -\frac{1}{n}\sum_{(j,r)\in\set{n}^2}p_{j,r}(1-p_{j,r}).
\end{align*}

\item For $(p_{j,r})\in\{0,1\}^{\set{n}\times \set{n}}$, we have
\begin{align*}
\gamma'
&=\frac{2}{n^2(n-1)}\sum_{(j,k)\in\set{n}_{\neq}^2}
\sum_{(r,s)\in\set{n}_{\neq}^2}p_{j,r}(1-p_{j,s})p_{k,s}(1-p_{k,r})\\
&=\frac{2}{n^2(n-1)}\sum_{(j,k)\in\set{n}^2}
\sum_{(r,s)\in\set{n}_{\neq}^2}p_{j,r}(1-p_{j,s})p_{k,s}(1-p_{k,r})\\
&\quad{}
-\frac{2}{n^2(n-1)}\sum_{j\in\set{n}}
\sum_{(r,s)\in\set{n}_{\neq}^2}p_{j,r}(1-p_{j,s})p_{j,s}(1-p_{j,r})\\
&=\frac{2}{n^2(n-1)}\sum_{(r,s)\in\set{n}_{\neq}^2}
\Bigl(\sum_{j\in\set{n}}p_{j,r}(1-p_{j,s})\Bigr)
\Bigl(\sum_{k\in\set{n}}p_{k,s}(1-p_{k,r})\Bigr)\\
&=\frac{2}{n-1}\sum_{(r,s)\in\set{n}_{\neq}^2}
\Bigl(\overline{p}_{\bfcdot,r}
-\frac{1}{n}\sum_{j\in\set{n}}p_{j,r}p_{j,s}\Bigr)
\Bigl(\overline{p}_{\bfcdot,s}
-\frac{1}{n}\sum_{j\in\set{n}}p_{j,r}p_{j,s}\Bigr).\qedhere
\end{align*}
\end{enumerate}
\end{proof}

\begin{proof}[Proof of Lemma \ref{l32864}]
From Lemma \ref{l387256}(\ref{l387256.b}), it follows that 
$\gamma'\leq \frac{2}{n}(\Var S_n-\lambda+\lambda^2)$. Therefore
\begin{align*}
A
\leq \lambda-\Var S_n+\frac{2}{n}(\Var S_n-\lambda+\lambda^2)
=B.
\end{align*}
From \eqref{e8164}, we get that $\gamma+\gamma'=\gamma''\geq 0$. 
Further, the Cauchy-Schwarz inequality implies that 
$\lambda^2=(\sum_{j=1}^n\overline{p}_{j,\bfcdot})^2
\leq n\sum_{j=1}^n\overline{p}_{j,\bfcdot}^2$. Therefore, 
using \eqref{e82165}, we obtain
\begin{align*}
B
&
=\frac{n-2}{n}\Bigl(\sum_{j=1}^n\overline{p}_{j,\bfcdot}^2+\gamma
\Bigr)+\frac{2\lambda^2}{n}
\leq \frac{n-2}{n}\Bigl(\sum_{j=1}^n\overline{p}_{j,\bfcdot}^2+\gamma
\Bigr)+2\sum_{j=1}^n\overline{p}_{j,\bfcdot}^2 \\
&\leq \Bigl(3-\frac{2}{n}\Bigr)
(\lambda-\Var S_n)
+2\gamma'
\leq 
\Bigl(3-\frac{2}{n}\Bigr)A.\qedhere
\end{align*}
\end{proof}


{\small 
\let\oldbibliography\thebibliography
\renewcommand{\thebibliography}[1]{\oldbibliography{#1}
\setlength{\itemsep}{0.8ex plus0.8ex minus0.8ex}} 
\linespread{1}
\selectfont
\bibliography{pardis_108.bib}
}
\end{document}